\documentclass[12pt]{amsart}
\usepackage{latexsym,amsmath,amssymb,graphicx,MnSymbol}
\usepackage{tikz}\usetikzlibrary{matrix,arrows,calc}
\newcommand{\BrMcInvAff}{Proposition~3.2}

\newcommand{\McReflections}{Theorem~9.6}
\newcommand{\McReducibleBowties}{Theorem~10.3}

\theoremstyle{plain}
\newtheorem{thm}{Theorem}[section]
\newtheorem{main}{Theorem}
\newtheorem{lem}[thm]{Lemma}
\newtheorem{prop}[thm]{Proposition}
\newtheorem{cor}[thm]{Corollary}
\newtheorem{conj}[thm]{Conjecture}
\theoremstyle{definition}
\newtheorem{defn}[thm]{Definition}
\newtheorem{exmp}[thm]{Example}
\newtheorem{rem}[thm]{Remark}
\newtheorem{quest}[thm]{Question}

\newcommand{\Z}{\mathbb{Z}}
\newcommand{\R}{\mathbb{R}}

\newcommand{\fix}{\textsc{Fix}}\newcommand{\ms}{\textsc{Min}}
\newcommand{\mov}{\textsc{Mov}}

\newcommand{\isom}{\textsc{Isom}}
\newcommand{\cay}{\textsc{Cay}}
\newcommand{\cox}{\textsc{Cox}}\newcommand{\art}{\textsc{Art}}
\newcommand{\midd}{\textsc{Mid}}
\newcommand{\gar}{\textsc{Gar}}\newcommand{\cryst}{\textsc{Cryst}}
\newcommand{\dart}{\textsc{Art${}^*$}}\newcommand{\dcox}{\textsc{Cox${}^*$}}
\newcommand{\sym}{\textsc{Sym}}
\newcommand{\proj}{\textrm{proj}}

\newcommand{\ccox}{W}
\newcommand{\ccryst}{C}
\newcommand{\chor}{H}
\newcommand{\cdiag}{D}
\newcommand{\cfac}{F}

\newcommand{\gart}{\ccox_w}
\newcommand{\ggar}{\ccryst_w}
\newcommand{\ghor}{\chor_w}
\newcommand{\gdiag}{\cdiag_w}
\newcommand{\gfac}{\cfac_w}

\newcommand{\wt}{\widetilde}
\newcommand{\onto}{\twoheadrightarrow}\newcommand{\into}{\hookrightarrow}

\newcommand{\bigjoin}{\bigvee}\newcommand{\bigmeet}{\bigwedge}
\newcommand{\join}{\vee}\newcommand{\meet}{\wedge}
\newcommand{\bc}{\begin{center}}\newcommand{\ec}{\end{center}}
\newcommand{\bt}{\begin{tabular}}\newcommand{\et}{\end{tabular}}

\newcommand{\drawEdge}[2]{\draw[-,thick] #1--#2;}
\newcommand{\drawDashEdge}[2]{\draw[-,dashed,thick] #1--#2;}
\newcommand{\drawArrow}[3]{
  \draw[-,thick] ($#1!.45!#2$)--($#1!.55!#2+#3$);
  \draw[-,thick] ($#1!.45!#2$)--($#1!.55!#2-#3$);
}
\newcommand{\drawTripleEdge}[3]{
  \draw[-,thick] #1--#2; 
  \draw[-,thick,yshift=.5mm] #1--#2; 
  \draw[-,thick,yshift=-.5mm] #1--#2;
  \drawArrow{#1}{#2}{#3}
}
\newcommand{\drawDoubleEdge}[3]{
  \draw[-,thick,yshift=.3mm] #1--#2; 
  \draw[-,thick,yshift=-.3mm] #1--#2;
  \drawArrow{#1}{#2}{#3}
}
\newcommand{\drawDashDoubleEdge}[3]{
  \draw[-,dashed,thick,yshift=.3mm] #1--#2; 
  \draw[-,dashed,thick,yshift=-.3mm] #1--#2;
  \drawArrow{#1}{#2}{#3}
}
\newcommand{\drawInfEdge}[2]{
  \drawDashEdge{#1}{#2} 
  \node[anchor=south] at ($#1!.5!#2$) {$\infty$};
}
\newcommand{\drawDotEdge}[2]{
  \fill ($#1!.3!#2$) circle (.4mm);
  \fill ($#1!.5!#2$) circle (.4mm);
  \fill ($#1!.7!#2$) circle (.4mm);
}
\newcommand{\drawRegDot}[1]{\fill #1 circle (.7mm);\fill[color=black!80] #1 circle (.5mm);}
\newcommand{\drawSpeDot}[1]{\fill #1 circle (1mm);\fill[color=white] #1 circle (.7mm);}
\newcommand{\drawESpDot}[1]{\fill #1 circle (1mm);\fill[color=red!70] #1 circle (.7mm);}

\begin{document}

\title[Euclidean Artin groups]{Artin groups of euclidean type}
\date{\today}

\author{Jon McCammond}
\address{Dept. of Math., University of California, Santa Barbara, CA 93106} 
\thanks{Partial support by the National Science Foundation is
  gratefully acknowledged}
\email{jon.mccammond@math.ucsb.edu}

\author{Robert Sulway}
\email{robertsulway@gmail.com}


\begin{abstract}
  This article resolves several long-standing conjectures about Artin
  groups of euclidean type.  Specifically we prove that every
  irreducible euclidean Artin group is a torsion-free centerless group
  with a decidable word problem and a finite-dimensional classifying
  space.  We do this by showing that each of these groups is
  isomorphic to a subgroup of a group with an infinite-type Garside
  structure.  The Garside groups involved are introduced here for the
  first time.  They are constructed by applying semi-standard
  procedures to crystallographic groups that contain euclidean Coxeter
  groups but which need not be generated by the reflections they
  contain.
\end{abstract}
\maketitle

\tableofcontents

Arbitrary Coxeter groups are groups defined by a particularly simple
type of presentation, but the central motivating examples that lead to
the general theory are the irreducible groups generated by reflections
that act geometrically (i.e. properly discontinuously and cocompactly
by isometries) on spheres and euclidean spaces.  Presentations for
these spherical and euclidean Coxeter groups are encoded in the
well-known Dynkin diagrams and extended Dynkin diagrams, respectively.

Arbitrary Artin groups are groups defined by a modified version of
these simple presentations, a definition designed to describe the
fundamental group of a space constructed from the complement of the
hyperplanes in a complexified version of the reflection arrangement
for the corresponding Coxeter group.  

The spherical Artin groups, i.e. the Artin groups corresponding to
Coxeter groups that act geometrically on spheres, have been well
understood ever since Artin groups themselves were introduced in 1972
by Pierre Deligne \cite{De72} and by Egbert Brieskorn and Kyoji Saito
\cite{BrSa72} in adjacent articles in the \emph{Inventiones}.  Given
the centrality of euclidean Coxeter groups in Coxeter theory and Lie
theory more generally, it has been somewhat surprising that the
structure of most euclidean Artin groups has remained mysterious for
the past forty plus years.

In this article we clarify the structure of all euclidean Artin groups
by showing that they are isomorphic to subgroups of a new class of
Garside groups that we believe to be of independent interest.  More
specifically we prove four main results.  The first establishes the
existence of a new class of Garside groups based on intervals in
crystallographic groups closely related to the irreducible euclidean
Coxeter groups.

\newcommand{\maingarside}{0}
\begin{main}[Crystallographic Garside groups]\label{main:garside}
  Let $W = \cox(\wt X_n)$ be an irreducible euclidean Coxeter group
  and let $R$ be its set of reflections.  For each Coxeter element
  $w\in W$ there exists a set of translations $T$ and a
  crystallographic group $\cryst(\wt X_n,w)$ containing $W$ with
  generating set $R \cup T$ so that the weighted factorizations of $w$
  over this expanded generating set form a balanced lattice.  As a
  consequence, this collection of factorizations define a group
  $\gar(\wt X_n,w)$ with a Garside structure of infinite-type.
\end{main}

The second shows that these crystallographic Garside groups contain
subgroups that we call dual euclidean Artin groups.

\newcommand{\mainsubgroup}{1}
\begin{main}[Dual Artin Subgroups]\label{main:subgroup}
  For each irreducible euclidean Coxeter group $\cox(\wt X_n)$ and for
  each choice of Coxeter element $w$, the Garside group $\gar(\wt
  X_n,w)$ is an amalgamated free product of explicit groups with the
  dual Artin group $\dart(\wt X_n,w)$ as one of its factors.  In particular,
  the dual Artin group $\dart(\wt X_n,w)$ injects into the Garside group
  $\gar(\wt X_n,w)$.
\end{main}

The third shows that this dual euclidean Artin group is isomorphic to
the corresponding Artin group.

\newcommand{\mainisomorphic}{2}
\begin{main}[Naturally isomorphic groups]\label{main:isomorphic}
  For each irreducible euclidean Coxeter group $W = \cox(\wt X_n)$ and
  for each choice of Coxeter element $w$ as the product of the
  standard Coxeter generating set $S$, the Artin group $A=\art(\wt
  X_n)$ and the dual Artin group $W_w = \dart(\wt X_n,w)$ are
  naturally isomorphic.
\end{main}

And finally, our fourth main result uses the Garside structure of the
crystallographic Garside supergroup to derive structural consequences
for its euclidean Artin subgroup.

\newcommand{\mainartin}{3}
\begin{main}[Euclidean Artin groups]\label{main:artin}
  Every irreducible euclidean Artin group $\art(\wt X_n)$ is a
  torsion-free centerless group with a solvable word problem and a
  finite-dimensional classifying space.
\end{main}

\begin{figure}
  $\begin{array}{ccccc}
    \art(\wt X_n) & \cong & \dart(\wt X_n,w) & \into & \gar(\wt X_n,w)\\
    \twoheaddownarrow & & \twoheaddownarrow & & \twoheaddownarrow\\
    \cox(\wt X_n) & \cong & \dcox(\wt X_n,w) & \into & \cryst(\wt X_n,w)
  \end{array}$
  \caption{For each Coxeter element $w$ in an irreducible euclidean
    Coxeter group of type $\wt X_n$ we define several related
    groups.\label{fig:maps}}
\end{figure}

The relations among these groups are shown in Figure~\ref{fig:maps}.
The notations in the middle column refer to the Coxeter group and the
Artin group as defined by their dual presentations.  These dual
presentations facilitate the connection between the Coxeter group
$\cox(\wt X_n)$ and the crystallographic group $\cryst(\wt X_n,w)$ and
between the Artin group $\art(\wt X_n)$ and the crystallographic
Garside group $\gar(\wt X_n,w)$.

Theorem~\ref{main:artin} represents a significant advance over what
was previous known.  In 1987, Craig Squier analyzed the euclidean
Artin groups with three generators: $\art(\wt A_2)$, $\art(\wt C_2)$
and $\art(\wt G_2)$ \cite{Squier87}.  His main technique was to
analyze the presentations as amalgamated products and HNN extensions
of known groups, a technique that does not appear to generalize to the
remaining groups.  The ones of type $A$ have been understood via a
semi-classical embedding $\art(\wt A_{n-1}) \into \art(B_n)$ into a
type $B$ spherical Artin group \cite{Al02,ChPe03,KePe02,tD98}.

More recently Fran\c{c}ois Digne used dual Garside structures to
successfully analyze the euclidean Artin groups of types $A$ and $C$
\cite{Digne06, Digne12}.  This article is the third in a series which
continues the investigation along these lines.  The first two papers
are \cite{BrMc-factor} and \cite{Mc-lattice} and there also is a
survey article \cite{Mc-artin-survey} that discusses the results in
all three papers.  The main result of \cite{Mc-lattice} was a negative
one: types $A$, $C$ and $G$ are the only euclidean types whose dual
presentations are Garside.  The results in this article show how to
overcome the deficiencies that arise in types $B$, $D$, $E$ and $F$.

\medskip\noindent \textbf{Overview:} The article is divided into four
parts.  Part I contains basic background definitions for posets,
Coxeter groups, intervals and Garside structures.  Part II introduces
an interesting discrete group generated by coordinate permutations and
translations by integer vectors whose structure is closely related by
the Coxeter and Artin groups of type $B$.  These ``middle groups'' and
the structure of their intervals play a major role in the proofs of
the main results.  Part III shifts attention to intervals in arbitrary
irreducible euclidean Coxeter groups and introduces various new groups
including the crystallographic groups and crystallographic Garside
groups mentioned above.  Part IV contains the proofs of our four main
results.

\part{Background}\label{part:background}
This part contains background material with one section focusing on
posets and Coxeter groups, another on intervals and Garside
structures.

\section{Posets and Coxeter groups}\label{sec:basic}

This section reviews some basic definitions for the sake of
completeness.  Our conventions follows \cite{Humphreys90}, \cite{EC1},
and \cite{DaPr02}.

\begin{defn}[Coxeter groups]\label{def:coxeter}
  A \emph{Coxeter group} is any group $W$ that can be defined by a
  presentation of the following form.  It has a standard finite
  generating set $S$ and only two types of relations.  For each $s\in
  S$ there is a relation $s^2=1$ and for each unordered pair of
  distinct elements $s,t\in S$ there is at most one relation of the
  form $(st)^m=1$ where $m = m(s,t) > 1$ is an integer.  When no
  relation involving $s$ and $t$ occurs we consider $m(s,t)=\infty$.
  A \emph{reflection} in $W$ is any conjugate of an element of $S$ and
  we use $R$ to denote the set of all reflections in $W$.  In other
  words, $R=\{ wsw^{-1} \mid s \in S, w \in W\}$.  This presentation
  is usually encoded in a labeled graph $\Gamma$ called a
  \emph{Coxeter diagram} with a vertex for each $s\in S$, an edge
  connecting $s$ and $t$ if $m(s,t)>2$ and a label on this edge if
  $m(s,t)>3$.  When every $m(s,t)$ is contained in the set
  $\{2,3,4,6,\infty\}$ the edges labeled $4$ and $6$ are replaced with
  double and triple edges, respectively.  The group defined by the
  presentation encoded in $\Gamma$ is denoted $W = \cox(\Gamma)$.  A
  Coxeter group is \emph{irreducible} when its diagram is connected.
\end{defn}

\begin{figure}
  \begin{tikzpicture}
    \begin{scope}[yshift=5.25cm,xshift=2cm]
      \node at (-3,0) {$\wt A_1$};
      \drawInfEdge{(0,0)}{(1,0)}
      \drawSpeDot{(1,0)}
      \drawESpDot{(0,0)}
    \end{scope}
    \begin{scope}[yshift=4cm,xshift=.5cm]
      \node at (-1.5,0) {$\wt A_n$};
      \drawEdge{(0,0)}{(2,0)}
      \drawDashEdge{(0,0)}{(2,.7)}
      \drawDashEdge{(4,0)}{(2,.7)}
      \drawEdge{(3,0)}{(4,0)}
      \drawDotEdge{(2,0)}{(3,0)}
      \foreach \x in {0,1,2,4} {\drawRegDot{(\x,0)}}
      \drawSpeDot{(2,.7)}
      \drawESpDot{(3,0)}
    \end{scope}
    \begin{scope}[yshift=3cm]
      \node at (-1,0) {$\wt C_n$};
      \drawDoubleEdge{(1,0)}{(0,0)}{(0,.1)}
      \drawEdge{(1,0)}{(2,0)}
      \drawDotEdge{(2,0)}{(3,0)}
      \drawEdge{(3,0)}{(4,0)}
      \drawDashDoubleEdge{(4,0)}{(5,0)}{(0,.1)}
      \foreach \x in {1,2,3,4} {\drawRegDot{(\x,0)}}
      \drawESpDot{(0,0)}
      \drawSpeDot{(5,0)}
    \end{scope}
    \begin{scope}[yshift=1.8cm]
      \node at (-1,0) {$\wt B_n$};
      \drawDoubleEdge{(0,0)}{(1,0)}{(0,.1)}
      \drawEdge{(1,0)}{(2,0)}
      \drawDotEdge{(2,0)}{(3,0)}
      \drawEdge{(3,0)}{(4,0)}
      \drawDashEdge{(4,0)}{(4.707,.707)}
      \drawEdge{(4,0)}{(4.707,-.707)}
      \drawRegDot{(4.707,-.707)}
      \foreach \x in {0,2,3,4} {\drawRegDot{(\x,0)}}
      \drawESpDot{(1,0)}
      \drawSpeDot{(4.707,.707)}
    \end{scope}
    \begin{scope}
      \node at (-1,0) {$\wt D_n$};
      \drawEdge{(.293,.707)}{(1,0)}
      \drawEdge{(.293,-.707)}{(1,0)}
      \drawEdge{(1,0)}{(2,0)}
      \drawDotEdge{(2,0)}{(3,0)}
      \drawEdge{(3,0)}{(4,0)}
      \drawDashEdge{(4,0)}{(4.707,.707)}
      \drawEdge{(4,0)}{(4.707,-.707)}
      \drawRegDot{(.293,.707)}
      \drawRegDot{(.293,-.707)}
      \drawRegDot{(4.707,-.707)}
      \foreach \x in {2,3,4} {\drawRegDot{(\x,0)}}
      \drawESpDot{(1,0)}
      \drawSpeDot{(4.707,.707)}
    \end{scope}
  \end{tikzpicture}
  \caption{Diagrams for the four infinite families.\label{fig:dynkin-families}}
\end{figure}
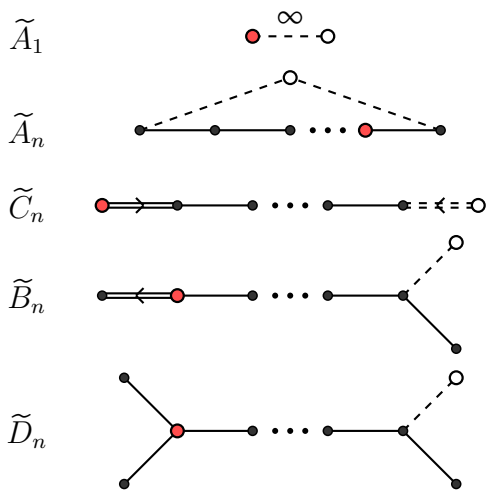

\begin{defn}[Artin groups]\label{def:artin}
  For each Coxeter diagram $\Gamma$ there is an \emph{Artin group}
  $\art(\Gamma)$ defined by a presentation with a relation for each
  two-generator relation in the standard presentation of
  $\cox(\Gamma)$.  More specifically, if $(st)^m=1$ is a relation in
  $\cox(\Gamma)$ then the presentation of $\art(\Gamma)$ has a
  relation that equates the two length~$m$ words that strictly
  alternate between $s$ and $t$.  Thus $(st)^2 = 1$ becomes $st=ts$,
  $(st)^3=1$ becomes $sts=tst$, $(st)^4=1$ becomes $stst=tsts$, etc.
  There is no relation when $m(s,t)$ is infinite.
\end{defn}

The general theory of Coxeter groups is motivated by those which act
\emph{geometrically} (i.e. properly discontinuously and cocompactly by
isometries) on spheres and euclidean spaces and they are classified by
the famous Dynkin diagrams and extended Dynkin diagrams, respectively.

\begin{defn}[Extended Dynkin diagrams]
  There are four infinite families and five sporadic examples of
  irreducible euclidean Coxeter groups.  The extended Dynkin diagrams
  for the infinite families, including the unusual $\wt A_1$ diagram,
  are shown in Figure~\ref{fig:dynkin-families} and the five sporadic
  examples are shown in Figure~\ref{fig:dynkin-sporadic}.  The large
  white dot connected to the rest of the diagram by dashed lines is
  the \emph{extending root} and the diagram with this dot removed is
  the ordinary Dynkin diagram for the corresponding spherical Coxeter
  group.  The large shaded dot is called the \emph{vertical root} of
  the diagram.  Its definition and meaning are discussed in
  Section~\ref{sec:horizontal}.
\end{defn}

\begin{figure}
  \begin{tikzpicture}
    \begin{scope}[yshift=5.8cm,xshift=1cm]
      \node at (-2,0) {$\wt G_2$};
      \drawTripleEdge{(0,0)}{(1,0)}{(0,.1)}
      \drawDashEdge{(1,0)}{(2,0)}
      \drawRegDot{(0,0)}
      \drawESpDot{(1,0)}
      \drawSpeDot{(2,0)}
    \end{scope}
    \begin{scope}[yshift=5.1cm]
      \node at (-1,0) {$\wt F_4$};
      \drawDoubleEdge{(1,0)}{(2,0)}{(0,.1)}
      \drawEdge{(0,0)}{(1,0)}
      \drawEdge{(2,0)}{(3,0)}
      \drawDashEdge{(3,0)}{(4,0)}
      \foreach \x in {0,1,3} {\drawRegDot{(\x,0)}}
      \drawESpDot{(2,0)}
      \drawSpeDot{(4,0)}
    \end{scope}
    \begin{scope}[yshift=3.4cm]
      \node at (-1,0) {$\wt E_6$};
      \drawEdge{(0,0)}{(4,0)}
      \drawEdge{(2,0)}{(2,1)}
      \drawDashEdge{(2,1)}{(3,1)}
      \foreach \x in {0,1,3,4} {\drawRegDot{(\x,0)}}
      \drawRegDot{(2,1)}
      \drawESpDot{(2,0)}
      \drawSpeDot{(3,1)}
    \end{scope}
    \begin{scope}[yshift=1.7cm]
      \node at (-1,0) {$\wt E_7$};
      \drawEdge{(0,0)}{(5,0)}
      \drawEdge{(2,0)}{(2,1)}
      \drawDashEdge{(0,0)}{(0,1)}
      \foreach \x in {0,1,3,4,5} {\drawRegDot{(\x,0)}}
      \drawRegDot{(2,1)}
      \drawESpDot{(2,0)}
      \drawSpeDot{(0,1)}
    \end{scope}
    \begin{scope}
      \node at (-1,0) {$\wt E_8$};
      \drawEdge{(0,0)}{(6,0)}
      \drawDashEdge{(6,0)}{(6,1)}
      \drawEdge{(2,0)}{(2,1)}
      \foreach \x in {0,1,3,4,5,6} {\drawRegDot{(\x,0)}}
      \drawRegDot{(2,1)}
      \drawESpDot{(2,0)}
      \drawSpeDot{(6,1)}
    \end{scope}
  \end{tikzpicture}
  \caption{Diagrams for the five sporadic examples.\label{fig:dynkin-sporadic}}
\end{figure}
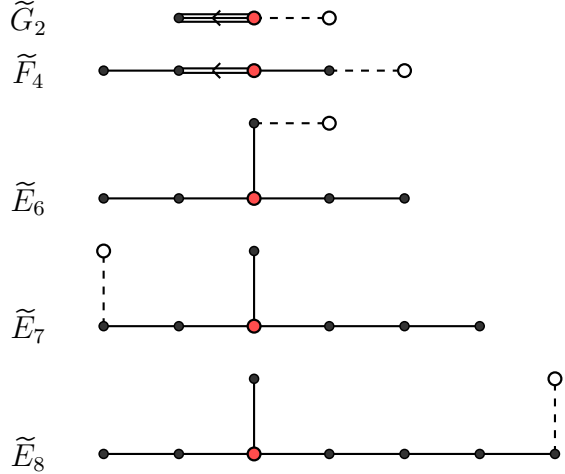

Since we do not need most of the heavy machinery developed to study
euclidean Coxeter groups, it suffices to loosely introduce some
standard terminology.

\begin{defn}[Simplices and tilings]
  One way to understand the meaning of the extended Dynkin diagrams is
  that they encode the geometry of a euclidean simplex $\sigma$ in
  which each dihedral angle is $\frac{\pi}{m}$ for some integer $m$.
  The vertices correspond to facets of $\sigma$ and the integer $m$
  associated to a pair of vertices encodes the dihedral angle between
  these facets.  The reflections that fix the facets of $\sigma$ then
  generate a group of isometries which tile euclidean space with
  copies of $\sigma$.  For example, the $\wt G_2$ diagram corresponds
  to a triangle in the plane with dihedral angles $\frac{\pi}{2}$,
  $\frac{\pi}{3}$, and $\frac{\pi}{6}$ and the corresponding tiling is
  shown in Figure~\ref{fig:g2-axis} on page~\pageref{fig:g2-axis}.
  The top dimensional simplices in this tiling are called
  \emph{chambers}.
\end{defn}

\begin{defn}[Roots and reflections]
  The \emph{root system $\Phi_{X_n}$} associated with the $\wt X_n$
  tiling is a collection of pairs of antipodal vectors called
  \emph{roots} which includes one pair $\pm \alpha$ normal to each
  infinite family of parallel hyperplanes and the length of $\alpha$
  encodes the consistent spacing between these hyperplanes.  The $G_2$
  root system is shown in Figure~\ref{fig:g2-roots}.  The $\wt X_n$
  tiling can be reconstructed from $\Phi_{X_n}$ root system as
  follows.  For each $\alpha \in \Phi_{X_n}$ and for each $k \in \Z$,
  let $H_{\alpha,k}$ be a hyperplane $\{ x \mid x \cdot \alpha = k\}$
  orthogonal to $\alpha$ and let $r_{\alpha,k}$ be the reflection that
  fixes $H_{\alpha,k}$ pointwise.  The chambers of the tiling are the
  connected components of the complement of the union of all such
  hyperplanes and the set $R = \{r_{\alpha,k}\}$ is the full set of
  reflections in the irreducible euclidean Coxeter group $W = \cox(\wt
  X_n)$.  The reflections $S \subset R$ that reflect in the facets of
  a single chamber $\sigma$ are a minimal generating set corresponding
  to the vertices of the $\wt X_n$ Dynkin diagram.
\end{defn}

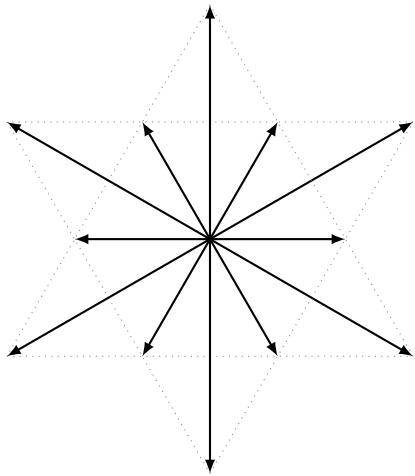
\begin{figure}
  \begin{tikzpicture}[>=latex]
    \foreach \x in {0,1,2,3,4,5} {
      \draw[-,color=black!50,dotted] (30+\x*60:1.8*1.732)--(150+\x*60:1.8*1.732);
      \draw[->,thick] (0,0)--(\x*60:1.8);
      \draw[->,thick] (0,0)--(30+\x*60:1.8*1.732);
    }
  \end{tikzpicture}
  \caption{The $G_2$ root system.\label{fig:g2-roots}}
\end{figure}

\begin{defn}[Coroots and translations]
  One consequence of these definitions is that a longer root
  corresponds to a family of hyperplanes that are more closely spaced.
  Let $t_\lambda$ be the translation produced by multiplying
  reflections associated with adjacent and parallel hyperplanes such
  as $r_{\alpha,k+1}$ and $r_{\alpha,k}$.  The translation vector is a
  multiple of $\alpha$ and one can compute $\lambda =
  \left(\frac{2}{\alpha \cdot \alpha}\right) \alpha$.  This vector is
  called the \emph{coroot} $\alpha^\vee$ corresponding to $\alpha$.
  In other words, $t_{\alpha^\vee} = r_{\alpha,k+1} r_{\alpha,k}$.
\end{defn}

We also record basic terminology for lattices and posets.

\begin{defn}[Posets]\label{def:poset}
  Let $P$ be a partially ordered set.  If $P$ contains both a minimum
  element and a maximum element then it is \emph{bounded}.  For each
  $Q\subset P$ there is an induced \emph{subposet} structure on $Q$ by
  restricting the partial order on $P$.  A subposet $C$ in which any
  two elements are comparable is called a \emph{chain} and its
  \emph{length} is $\vert C \vert -1$.  Every finite chain is bounded
  and its maximum and minimum elements are its \emph{endpoints}.  If a
  finite chain $C$ is not a subposet of a strictly larger finite chain
  with the same endpoints, then $C$ is \emph{saturated}.  Saturated
  chains of length~$1$ are called \emph{covering relations}.  If every
  saturated chain in $P$ between the same pair of endpoints has the
  same finite length, then $P$ is \emph{graded}.  There is also a
  weighted version where one defines a weight or length to each
  covering relation and calls $P$ \emph{weighted graded} when every
  saturated chain in $P$ between the same pair of endpoints has the
  same total weight.  When varying weights are introduced they shall
  always be \emph{discrete} in the sense that the set of all weights
  is a discrete subset of the positive reals bounded away from zero.
  The \emph{dual} $P^*$ of a poset $P$ has the same underlying set but
  the order is reversed, and a poset is \emph{self-dual} when it and
  its dual are isomorphic.
\end{defn}

\begin{defn}[Lattices]\label{def:lattice}
  Let $Q$ be any subset of a poset $P$.  A lower bound for $Q$ is any
  $p\in P$ with $p \leq q$ for all $q\in Q$.  When the set of lower
  bounds for $Q$ has a unique maximum element, this element is the
  \emph{greatest lower bound} or \emph{meet} of $Q$.  Upper bounds and
  the \emph{least upper bound} or \emph{join} of $Q$ are defined
  analogously.  The meet and join of $Q$ are denoted $\bigmeet Q$ and
  $\bigjoin Q$ in general and $u \meet v$ and $u \join v$ if $u$ and
  $v$ are the only elements in $Q$.  When every pair of elements has a
  meet and a join, $P$ is a \emph{lattice} and when every subset has a
  meet and a join, it is a \emph{complete lattice}.
\end{defn}

\begin{figure}
  \begin{tikzpicture}
    \coordinate (one) at (0,1.2);
    \coordinate (zero) at (0,-1.2);
    \coordinate (a) at (-.6,.4);
    \coordinate (b) at (.6,.4);
    \coordinate (c) at (-.6,-.4);
    \coordinate (d) at (.6,-.4);
    \draw[-,thick] (b)--(one)--(a)--(c)--(zero)--(d);
    \draw[-,thick] (a)--(d)--(b)--(c);
    \fill (one) circle (.6mm) node[anchor=south] {$1$};
    \fill (a) circle (.6mm) node[anchor=east] {$a$};
    \fill (b) circle (.6mm) node[anchor=west] {$b$};
    \fill (c) circle (.6mm) node[anchor=east] {$c$};
    \fill (d) circle (.6mm) node[anchor=west] {$d$};
    \fill (zero) circle (.6mm) node[anchor=north] {$0$};
  \end{tikzpicture}
  \caption{A bounded graded poset that is not a
    lattice.\label{fig:non-lattice}}
\end{figure}
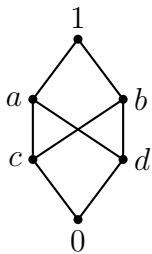

\begin{defn}[Bowties]\label{def:bowtie}
  Let $P$ be a poset.  A \emph{bowtie} in $P$ is a $4$-tuple of
  distinct elements $(a,b:c,d)$ such that $a$ and $b$ are minimal
  upper bounds for $c$ and $d$ and $c$ and $d$ are maximal lower
  bounds for $a$ and $b$.  The name reflects the fact that when edges
  are drawn to show that $a$ and $b$ are above $c$ and $d$, the
  configuration looks like a bowtie.  See
  Figure~\ref{fig:non-lattice}.
\end{defn}

In \cite[Proposition~1.5]{BrMc10} Tom Brady and the first author noted
that a bounded graded poset $P$ is a lattice if and only if $P$
contains no bowties.  The same result holds, with the same proof, when
$P$ is graded with respect to a discrete weighting of its covering
relations.  We reproduce the proof for completeness.

\begin{prop}[Lattice or bowtie]\label{prop:bowtie}
  If $P$ is a bounded poset that is graded with respect to a set of
  discrete weights, then $P$ is a lattice if and only if $P$ contains
  no bowties.
\end{prop}

\begin{proof}
  If $P$ contains a bowtie $(a,b:c,d)$, then $c$ and $d$ have no join
  and $P$ is not a lattice.  In the other direction, suppose $P$ is
  not a lattice because $x$ and $y$ have no join.  An upper bound
  exists because $P$ is bounded, and a minimal upper bound exists
  because $P$ is weighted graded.  Thus $x$ and $y$ must have more
  than one minimal upper bound.  Let $a$ and $b$ be two such minimal
  upper bounds and note that $x$ and $y$ are lower bounds for $a$ and
  $b$.  If $c$ is a maximal lower bound of $a$ and $b$ satisfying $c
  \geq x$ and $d$ is a maximal lower bound of $a$ and $b$ satisfying
  $d \geq y$, then $(a,b:c,d)$ is a bowtie.  We know that $a$ and $b$
  are minimal upper bounds of $c$ and $d$ and that $c$ and $d$ are
  distinct since either failure would create an upper bound of $x$ and
  $y$ that contradicts the minimality of $a$ and $b$.  When $x$ and
  $y$ have no meet, the proof is analogous.
\end{proof}

We conclude with a remark about subposets.  Notice that bowties remain
bowties in induced subposets.  Thus if $P$ is not a lattice because it
contains a bowtie $(a,b:c,d)$ and $Q$ is any subposet that contains
all four of these elements, then $Q$ is also not a lattice since it
contains the same bowtie.

\section{Intervals and Garside structures}\label{sec:int-garside}

As mentioned in the introduction, attempts to understand euclidean
Artin groups of euclidean type directly have only met with limited
success.  The most promising progress has been by Fran\c{c}ois Digne
and his approach is closely related to the dual presentations derived
from an interval in the corresponding Coxeter group \cite{Mc-lattice}.
We first recall how a group with a fixed generating set naturally acts
on a graph and how assigning discrete weights to its generators turns
this graph into a metric space invariant under the group action.

\begin{defn}[Marked groups]
  A \emph{marked group} is a group $G$ with a fixed generating set $S$
  which, for convenience, we assume is symmetric and injects into $G$.
  The (right) \emph{Cayley graph of $G$ with respect to $S$} is a
  labeled directed graph denoted $\cay(G,S)$ with vertices indexed by
  $G$ and edges indexed by $G \times S$.  The edge $e_{(g,s)}$ has
  \emph{label} $s$, it starts at $v_g$ and ends at $v_{g'}$ where $g'
  = g\cdot s$. There is a natural faithful, vertex-transitive, label
  and orientation preserving left action of $G$ on its right Cayley
  graph and these are the only graph automorphisms that preserve
  labels and orientations.
\end{defn}

\begin{defn}[Weights]
  Let $S$ be a generating set for a group $G$.  We say $S$ is a
  \emph{weighted generating set} if its elements are assigned positive
  weights bounded away from $0$ that form a discrete subset of the
  positive reals.  The elements $s$ and $s^{-1}$ should, of course,
  have the same weight.  For finite generating sets discreteness and
  the lower bound are automatic but these are important restrictions
  when $S$ is infinite.  One can always use a \emph{trivial weighting}
  which assigns the same weight to each generator.  When $G$ is
  generated by a weighted set $S$, its Cayley graph can be made into a
  metric space where the length of each edge is its weight.  The
  length of a combinatorial path in the Cayley group is then the sum
  of the weights of its edges and the distance between two vertices is
  the minimum length of such a combinatorial path.  For infinite
  generating sets the lower bound on the weights can be used to bound
  on the number of edges involved in a minimum length path and the
  discreteness condition ensures that the infimum of these path
  lengths is actually achieved by some path.
\end{defn}

In any metric space, one can define the notion of an interval.

\begin{defn}[Intervals in metric spaces]
  Let $x$, $y$ and $z$ be points in a metric space $(X,d)$.  We say
  $z$ is \emph{between} $x$ and $y$ if the triangle inequality is an
  equality: $d(x,z) + d(z,y) = d(x,y)$.  The \emph{interval} $[x,y]$
  is the collection of points between $x$ and $y$, and note that this
  includes both $x$ and $y$.  Intervals can also be endowed with a
  partial ordering by defining $u \leq v$ when $d(x,u) + d(u,v) +
  d(v,y) = d(x,y)$.
\end{defn}

We are interested in intervals in groups.

\begin{defn}[Intervals in groups]  
  Let $G$ be a group with a fixed symmetric discretely weighted
  generated set and let $d(g,h)$ denote the distance between $v_g$ and
  $v_h$ in the corresponding metric Cayley graph.  Note that the
  symmetry assumption on the generating set allows us to restrict
  attention to directed paths.  From this metric on $G$ we get bounded
  intervals with a weighted grading: for $g,h \in G$, the
  \emph{interval} $[g,h]^G$ is the poset of group elements between $g$
  and $h$ with $g' \in [g,h]^G$ when $d(g,g') + d(g',h) = d(g,h)$ and
  $g' \leq g''$ when $d(g,g') + d(g',g'') + d(g'',h) = d(g,h)$.  In
  this article we include the superscript $G$ as part of the notation
  since we often consider similar intervals in closely related groups.
\end{defn}

\begin{rem}[Intervals in Cayley graphs]
  The interval $[g,h]^G$ is a bounded poset with discrete levels whose
  Hasse diagram is embedded as a subgraph of the weighted Cayley graph
  $\cay(G,S)$ as the union of all minimal length directed paths from
  $v_g$ to $v_h$.  This is because $g'\in [g,h]^G$ means $v_{g'}$ lies
  on some minimal length path from $v_g$ to $v_h$ and $g' < g''$ means
  that $v_{g'}$ and $v_{g''}$ both occur on a common minimal length
  path from $v_g$ to $v_h$ with $v_{g'}$ occurring before $v_{g''}$.
  Because the structure of such a poset can be recovered from its
  Hasse diagram, we let $[g,h]^G$ denote the edge-labeled directed
  graph that is visible inside $\cay(G,S)$.  The left action of a
  group on its right Cayley graph preserves labels and distances.
  Thus the interval $[g,h]^G$ is isomorphic (as a labeled oriented
  directed graph) to the interval $[1,g^{-1}h]^G$.  In other words,
  every interval in the Cayley graph of $G$ is isomorphic to one that
  starts at the identity.  We call $g^{-1}h$ the \emph{type} of the
  interval $[g,h]^G$ and note that intervals are isomorphic if and
  only if they have the same type.
\end{rem}

Intervals in groups can be used to construct new groups.

\begin{defn}[Interval groups]\label{def:interval-gps}
  Let $G$ be a group generated by a weighted set $S$ and let $g$ and
  $h$ be distinct elements in $G$.  The \emph{interval group}
  $G_{[g,h]}$ is defined as follows.  Let $S_0$ be the elements of $S$
  that actually occurs as labels of edges in $[g,h]^G$.  The group
  $G_{[g,h]}$ has $S_0$ as its generators and we impose all relations
  that are visible as closed loops inside the portion of the Cayley
  graph of $G$ that we call $[g,h]^G$.  The elements in $S \setminus
  S_0$ are not included since they do not occur in any relation.  More
  precisely, if they were included as generators, they would generate
  a free group that splits off as a free factor.  Thus it is
  sufficient to understand the group defined above.  Next note that
  this group structure only depends on the type of the interval so it
  is sufficient to consider interval groups of the form $G_{[1,g]}$.
  For these groups we simplify the notation to $G_g$ and say that
  $G_g$ is the interval group obtained by \emph{pulling $G$ apart at
    $g$}.
\end{defn}

The interval $[1,g]^G$ implicitly encodes a presentation of $G_g$ and
various explicit presentations can be found in \cite{Mc-lattice} and
\cite{McCammond-cont-braids}.  Dual Artin groups are examples of
interval groups.

\begin{defn}[Dual Artin groups]\label{def:dual-artin}
  Let $W = \cox(\Gamma)$ be a Coxeter group with standard generating
  set $S$ and let $R$ be the full set of reflections with a trivial
  weighting.  For any fixed total ordering of the elements of $S$, the
  product of these generators in this order is called a \emph{Coxeter
    element} and for each Coxeter element $w$ there is a dual Artin
  group defined as follows.  Let $[1,w]^W$ be the interval in the left
  Cayley graph of $W$ with respect to $R$ and let $R_0 \subset R$ be
  the subset of reflections that actually occur in some minimal length
  factorizations of $w$.  The \emph{dual Artin group with respect to
    $w$} is the group $W_w = \dart(\Gamma,w)$ generated by $R_0$ and
  subject only to those relations that are visible inside the interval
  $[1,w]^W$.
\end{defn}

An explicit presentation for the dual $\wt G_2$ Artin group is given
at the end of Section~\ref{sec:euc-intervals}.

\begin{rem}[Artin groups and dual Artin groups]
  In general the relationship between the Artin group $\art(\Gamma)$
  and the dual Artin group $\dart(\Gamma,w)$ is not yet completely
  clear.  It is straightforward to show using the Tits representation
  that the product of the elements in $S$ that produce $w$ is a
  factorization of $w$ into reflections of minimum length which means
  that this factorization describes a directed path in $[1,w]^W$.  As
  a consequence $S$ is a subset of $R_0$.  Moreover, the standard
  Artin relations are consequences of relations visible in $[1,w]^W$
  (as illustrated in \cite{BrMc00} and in Example~\ref{ex:dihedral})
  so that the injection of $S$ into $R_0$ extends to a group
  homomorphism from $\art(\Gamma)$ to $\dart(\Gamma,w)$.  See also
  Proposition~\ref{prop:onto-dart}. When this homomorphism is an
  isomorphism, we say that the interval $[1,w]^W$ encodes a \emph{dual
    presentation} of $\art(\Gamma)$.
\end{rem}

To date, every dual Artin group that has been successfully analyzed is
isomorphic to the corresponding Artin group and as a consequence its
group structure is independent of the Coxeter element $w$ used in its
construction.  In particular, this is known to hold for all spherical
Artin groups \cite{Be03, BrWa02a} and we prove it here for all
euclidean Artin groups as our third main result.  It is precisely
because this assertion has not been proved in full generality that
dual Artin groups deserve a separate name.  The following example
illustrates the relationship between Artin group presentations and
dual presentations.

\begin{exmp}[Dihedral Artin groups]\label{ex:dihedral}
  The spherical Coxeter groups with two generators are the dihedral
  groups.  Let $W$ be the dihedral group of order $10$ with Coxeter
  presentation $\langle a,b \mid a^2 = b^2 = (ab)^5 = 1 \rangle$ where
  $a$ and $b$ are reflections of $\R^2$ through the origin with an
  angle of $\pi/5$ between their fixed lines.  The corresponding Artin
  group has presentation $\langle a,b \mid ababa = babab \rangle$.
  The set $S = \{a,b\}$ is a standard generating set for $W$ and the
  set $R = \{a,b,c,d,e\}$ is its full set of reflections where these
  are the five reflections in $W$ in cyclic order.  The Coxeter
  element $w = ab$ is a $2\pi/5$ rotation and its minimum length
  factorizations over $R$ are $ab$, $bc$, $cd$, $de$ and $ea$.  The
  dual Artin group has presentation $\langle a,b,c,d,e \mid ab = bc =
  cd = de = ea \rangle$. Systematically eliminating $c$, $d$ and $e$
  recovers the original Artin group presentation.
\end{exmp}

The dual presentations for the spherical Artin groups were introduced
and studied by David Bessis \cite{Be03} and by Tom Brady and Colum
Watt \cite{BrWa02a}.  Here we pause to record one technical fact about
Coxeter elements in the corresponding spherical Coxeter groups.

\begin{prop}[Spherical Coxeter elements]\label{prop:sph-cox-elt}
  Let $w_0$ be a Coxeter element for a spherical Coxeter group $W_0 =
  \cox(X_n)$ and let $R$ be its set of reflections.  For every $r \in
  R$ there is a chamber in the corresponding spherical tiling and an
  ordering on the reflections fixing its facets so that (1) the
  product of these reflections in this order is $w_0$ and (2) the
  leftmost reflection in the list is $r$.
\end{prop}

One reason that dual presentations of Artin groups are of interest is
that they satisfy almost all of the requirements of a Garside
structure.  In fact, there is only one property that they might lack.

\begin{prop}[Garside structures]\label{prop:lattice-garside}
  Let $G$ be a group with a symmetric discretely weighted generating
  set that is closed under conjugation.  If for some element $g$ the
  weighted interval $[1,g]^G$ is a lattice, then the group $G_g$ is a
  Garside group.  In particular, if $W = \cox(\Gamma)$ is a Coxeter
  group generated by its full set of reflections with Coxeter element
  $w$ and the interval $[1,w]^W$ is a lattice, then the dual Artin
  group $W_w = \dart(\Gamma,w)$ is a Garside group.
\end{prop}

The reader should note that we are using ``Garside structure'' and
``Garside group'' in the expanded sense of Digne \cite{Digne06,
  Digne12} rather than the original definition that requires the
generating set to be finite.  The discreteness of the grading of the
interval substitutes for finiteness of the generating set.  In
particular, the discreteness of the grading forces the standard
Garside algorithms to terminate.  The standard proofs are otherwise
unchanged.  Proposition~\ref{prop:lattice-garside} was stated by David
Bessis in \cite[Theorem~0.5.2]{Be03}, except for the shift from finite
to infinite discretely weighted generating sets.  For a more detailed
discussion see \cite{Be03} and particularly the book
\cite{DDGKM-garside}.  Interval groups appear in
\cite{DehornoyDigneMichel13} and in \cite[Chapter VI]{DDGKM-garside}
as the ``germ derived from a groupoid''.  The terminology is different
but the translation is straightforward.  When an interval such as
$[1,w]$ is a lattice and it is used to construct a Garside group, the
interval $[1,w]$ itself embeds in the Cayley graph of the new group
$G$ and the element $w$, viewed as an element in $G$, is called a
\emph{Garside element}.  Being a Garside group has many consequences.

\begin{thm}[Consequences]\label{thm:consequences}
  If $G$ is a group with a Garside structure in the expanded sense of
  Digne, then its elements have normal forms and it has a
  finite-dimensional classifying space whose dimension is equal to the
  length of the longest chain in its defining interval.  As a consequence,
  $G$ has a decidable word problem and it is torsion-free.
\end{thm}

\begin{proof}
  The initial consequences follow from \cite{DePa99} and
  \cite{ChMeWh04} with minor modifications to allow for infinite
  discretely weighted generating sets, and the latter ones are
  immediate corollaries.
\end{proof}

A detailed description of the Garside normal form is never needed, but
we give a coarse description sufficient to state a key property of
elements that commute with the Garside element.

\begin{defn}[Normal forms]
  Let $G$ be a Garside group in the expanded sense used here and with
  Garside element $w$. The elements in the interval $[1,w]$ are called
  \emph{simple elements}.  For every $u \in G$ there is an integer $n$
  and simple elements $u_i$ such that $u = w^n u_1 u_2 \cdots u_k$.
  If we impose a few additional conditions, the integer $n$ and the
  simples $u_i$ are uniquely determined by $u$ and this expression is
  called its (left-greedy) \emph{normal form} $NF(u)$.  Note that the
  integer $n$ might be negative.  When this happens, it indicates that
  the word $u$ does not belong to the positive monoid generated by the
  simple elements.  The value of $n$ is the smallest integer such that
  $w^{-n} u$ lies in this positive monoid.
\end{defn}

One consequence of being a Garside group is that the set of simples is
closed under conjugation by $w$.  In fact, conjugation by $w$ is a
lattice isomorphism (but one that typically does not preserve
edge-labels) sending each simple to the left complement of its left
complement.  In particular, the simple by simple conjugation of the
normal form for $u$ remains in normal form, its product is $u^w$ and
by the uniqueness of normal forms, this must be the normal form for
$u^w$.  In particular, this proves the following.

\begin{prop}[Normal forms]\label{prop:nf}
  Let $G$ be a Garside group with Garside element $w$.  For each $u
  \in G$, the Garside normal form of $u^w$ is obtained by conjugating
  each simple in the Garside normal form for $u$.  In other words, if
  $NF(u) = w^n u_1 u_2 \cdots u_k$ then $NF(u^w) = w^n u_1^w u_2^w
  \cdots u_k^w$.  In particular, an element in $G$ commutes with $w$
  if and only if its normal form is built out of simples that commute
  with $w$.
\end{prop}

There is one final fact about Garside structures that we need in the
later sections, and that is an elementary observation about nicely
situated sublattices of lattices and how they relate to normal forms.

\begin{prop}[Injective maps]\label{prop:inj}
  Let $G'$ be a subgroup of $G$ and let $S'$ and $S$ be their
  conjugacy closed generating sets with $S' \subset S$.  If $w$ is an
  element of $G'$ and there is a weighting on $S$ such that (1) both
  $[1,w]^{G'}$ and $[1,w]^G$ are lattices and (2) the inclusion map
  $[1,w]^{G'} \into [1,w]^G$ is a lattice homomorphism preserving
  meets and joins, then the interval groups $G'_w$ and $G_w$ are both
  Garside groups and the natural map $G'_w \to G_w$ is an injection.
\end{prop}

\begin{proof}
  First note that there is a natural map from $G'_w$ to $G_w$ because
  the relations defining $G'_w$ are included among the relations that
  define $G_w$.  For injectivity, let $u$ be a nontrivial element of
  $G'_w$ with normal form $NF(u) = w^n u_1 u_2 \cdots u_k$.  Suppose
  we view this as an expression representing an element of $G_w$.  The
  fact that the inclusion of the smaller interval into the larger one
  preserves meets and joins means that this expression remains in
  normal form in this new context.  Therefore the image of $u$ in
  $G_w$ is also nontrivial and the map is an injection.
\end{proof}

\part{Middle groups}\label{part:middle}
This part focuses on a series of elementary groups that we call
``middle groups'' and it establishes their key properties.

\section{Permutations and translations}\label{sec:perm-trans}

The discrete group of euclidean isometries generated by all coordinate
permutations and all translations by vectors with integer coordinates
is a group that plays an important role in the proofs of our main
results.  In this section we record its basic properties and relate it
to the spherical Coxeter group $\cox(B_n)$ which encodes the
symmetries of the $n$-cube.

\begin{defn}[Cubical symmetries]\label{def:n-cube}
  Let $[-1,1]^n$ denote the points in $\R^n$ where every coordinate
  has absolute value at most $1$.  This $n$-dimensional cube of side
  length $2$ centered at the origin has isometry group $\cox(B_n)$
  also called the \emph{signed symmetric group}. It has $n^2$
  reflection symmetries.  We give these reflections nonstandard names
  based on an alternative realization of this group described below.
  Let $r_{ij}$ be the reflection which switches the $i$-th and $j$-th
  coordinates fixing the hyperplane $x_i = x_j$ and let $t_i$ denote
  the reflection which changes the sign of the $i$-th coordinate
  fixing the hyperplane $x_i = 0$.  We call these collections
  $\mathcal{R}$ and $\mathcal{T}$ respectively.  Together they
  generate $\cox(B_n)$, but they are neither a minimal generating set
  nor all of the reflections.  The remaining reflections are obtained
  by conjugation.  Conjugating $r_{ij}$ by $t_i$, for example,
  produces an isometry which switches the $i$-th and $j$-th
  coordinates and changes both signs fixing the hyperplane $x_i =
  -x_j$.
\end{defn}

The unusual names for the reflections are explained by an alternative
geometric realization of $\cox(B_n)$ as isometries of an $n$-torus.

\begin{defn}[Toroidal symmetries]\label{def:n-torus}
  Let $T^n$ be the $n$-dimensional torus formed by identifying
  opposite sides of the $n$-cube $[-1,1]^n$ and note that the
  previously defined action of action of $\cox(B_n)$ the $n$-cube
  descends to $T^n$ and it permutes the $2^n$ special points with
  every coordinate equal to $\pm \frac{1}{2}$.  In fact, the action of
  $\cox(B_n)$ on these $2^n$ special points is faithful.  A new action
  of $\cox(B_n)$ on $T^n$ is obtained by leaving the action of the
  $r_{ij} \in \mathcal{R}$ unchanged and by replacing the
  ``reflection'' $t_i \in \mathcal{T}$ with a ``translation'' which
  adds $1$ to the $i$-th coordinate mod $2$.  This has the net effect
  of switching the sign of the $i$-th coordinate of each special point
  because the translation $x \mapsto x+1$ in $\R/2\Z$ switches
  $\frac12$ and $-\frac12$.  Since the elements in $\mathcal{R} \cup
  \mathcal{T}$ act on the $2^n$ special points as before, they
  generate the same group up to isomorphism.
\end{defn}

The group we wish to discuss is generated by lifts of these toriodal
isometries to all of $\R^n$.

\begin{defn}[Permutations and translations]\label{def:perm-trans}
  Let $r_{ij}$ act on $\R^n$ as before and let $t_i$ denote the
  translation which adds $1$ to the $i$-th coordinate leaving the
  others unchanged.  The reflections $\mathcal{R} = \{ r_{ij} \}$ and
  the translations $\mathcal{T} = \{ t_i \}$ generate a group
  $\midd(B_n)$ that we call the \emph{middle group} or more formally
  the \emph{annular symmetric group}.  The names are explained below.
  The isometries in $\mathcal{R}$ generates the symmetric group
  $\sym_n$ and the isometries in $\mathcal{T}$ generate a free abelian
  group $\Z^n$.  Moreover, because the $\Z^n$ subgroup generated by
  the translations is normalized by the permutations in $\sym_n$ with
  trivial intersection, the full group $\midd(B_n)$ has the structure
  of a semidirect product $\Z^n \rtimes \sym_n$.  Every element of
  $\midd(B_n)$ can be written uniquely in the form $t_\lambda r_\pi$
  where $\lambda\in \Z^n$ is a vector with integer entries and $\pi
  \in \sym_n$ is a permutation.
\end{defn}

\begin{defn}[Reflections]\label{def:reflections}
  As in $\cox(B_n)$ there are other reflections in $\midd(B_n)$
  obtained by conjugation.  The basic translations in $\mathcal{T}$
  are closed under conjugation but infinitely many new reflections are
  added to $\mathcal{R}$ when we close this set under conjugation.
  For example, $t_1 r_{12} t_1^{-1}$ is a reflection whose fixed
  hyperplane is parallel to that of $r_{12}$.  We call this reflection
  $r_{12}(1)$.  More generally, for each integer $k$ we define
  $r_{ij}(k) := t_i^k r_{ij} t_i^{-k} = t_j^{-k} r_{ij} t_j^k$.  The
  original reflections are $r_{ij} = r_{ij}(0)$.  Let $\mathcal{R}'$
  denote the set of all these reflections.  The set $\mathcal{R}' \cup
  \mathcal{T}$ is called the \emph{full generating set} of
  $\midd(B_n)$.
\end{defn}

The center of a middle group is easy to compute.

\begin{prop}[Center]\label{prop:midd-center}
  The center of the middle group $\midd(B_n)$ is an infinite cyclic
  subgroup generated by the pure translation $t_{\mathbf{1}} =
  \prod_{i=1}^n t_i$.
\end{prop}

\begin{proof}
  Let $u = t_\lambda r_\pi$ be an element in the center.  If $\pi$ is
  a nontrivial permutation and $i$ is an index that is moved by $\pi$
  then $u$ conjugates $t_i$ to $t_{\pi(i)}$, contradiction.  Thus $u$
  must be a pure translation $t_\lambda$.  In order for $t_\lambda$ to
  be central $\lambda$ must be orthogonal to all of the roots of the
  reflections and thus in the direction $\mathbf{1} = \langle 1^n
  \rangle$.  In particular, the center is contained in the infinite
  cyclic subgroup generated by $t_{\mathbf{1}} = \prod_{i=1}^n t_i$
  which adds $1$ to every coordinate.  Conversely, $t_{\mathbf{1}}$
  commutes with every element of $M$.
\end{proof}

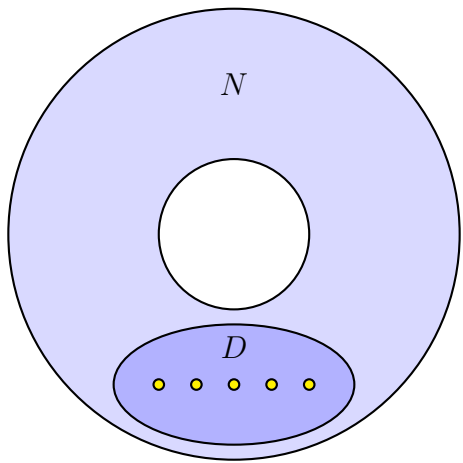
\begin{figure}
  \begin{tikzpicture}
    \fill[color = blue!15!white] (0,0) circle (30mm);
    \fill[color = white] (0,0) circle (10mm);
    \fill[color = blue!30!white] (0,-20mm) ellipse (16mm and 8mm);
    \draw[thick] (0,0) circle (30mm);
    \draw[thick] (0,0) circle (10mm);
    \draw[thick] (0,-20mm) ellipse (16mm and 8mm);
    \foreach \x in {-2,-1,0,1,2} {
      \fill[color=yellow] (.5*\x,-20mm) circle (.6mm);
      \draw[thick] (.5*\x,-20mm) circle (.7mm);
    }
    \draw (0,20mm) node {$N$};
    \draw (0,-15mm) node {$D$};
  \end{tikzpicture}
  \caption{The configuration that explains the name annular symmetric
    group.\label{fig:ann-sym}}
\end{figure}

We call the spherical Artin group $\art(B_n)$ an \emph{annular braid
  group} because it is the \emph{braid group of the annulus} in the
sense of Birman \cite[p.11]{Bi74}.  The analogous definition of an
\emph{annular symmetric group} leads to an alternative perspective on
the group $\midd(B_n)$.

\begin{defn}[Annular symmetric groups]\label{def:annular-gps}
  Let $N$ be an annulus, let $D$ be a disk contained in $N$ and let
  $p_i$, $i \in \{1,2,\ldots,n\}$ be a set of $n$ distinct points in
  $D$.  See Figure~\ref{fig:ann-sym}.  The annular braid group is
  defined as the fundamental group of the configuration space of $n$
  unordered distinct points in the annulus with this configuration as
  its base point.  It keeps track of the way in which the points braid
  around each other well as how much they wind around the annulus.  An
  \emph{annular symmetric group} ignores the braiding and only keeps
  track of how the points permute and wind around the annulus.  More
  concretely, if we define $r_{ij}$ as the motion which swaps $p_i$
  and $p_j$ without leaving the disk $D$ and define $t_i$ as the
  motion which wraps the point $p_i$ once around the annulus $N$ in
  the direction considered positive in its fundamental group, then the
  elements of this group can be identified with the elements of
  $\midd(B_n)$ and its normal form $t_\lambda r_\pi$ can be recovered
  as follows.  The permutation $\pi$ records the permutation of the
  points and the vector $\lambda$ is a tuple of winding numbers
  obtained by viewing the path of the point $p_i$ as a (near) loop
  that starts and ends in the disk $D$ and letting its winding number
  be the $i$-th coordinate of $\lambda$.
\end{defn}

The name ``middle group'' refers to its close connections with 
various other groups as shown in Figure~\ref{fig:middle-maps}.

\begin{rem}[Affine braid groups]\label{rem:affine-braid-gps}
  Since the groups $\cox(A_{n-1})$ and $\art(A_{n-1})$ are symmetric
  groups and braid groups, we call their natural euclidean extensions,
  the \emph{euclidean symmetric group} $\cox(\wt A_{n-1})$ and the
  \emph{euclidean braid group} $\art(\wt A_{n-1})$.  The name ``affine
  braid group'' often appears in the literature but its meaning is not
  stable.  For geometric group theorists it refers to the euclidean
  braid group $\art(\wt A_{n-1})$ \cite{ChPe03} but for representation
  theorists it refers to the annular braid group $\art(B_n)$
  \cite{OrRa07}.  Our alternative names aim to limit this potential
  confusion.  The adjective ``euclidean'' also highlights that the
  Coxeter group preserves lengths and angles.
\end{rem}

\begin{figure}
  $\begin{array}{ccccc}
    \art(\wt A_{n-1}) & \into & \art(B_n) & \onto &\Z\\
    \twoheaddownarrow & & \twoheaddownarrow & & \parallel\\
    \cox(\wt A_{n-1}) & \into & \midd(B_n) & \onto & \Z\\
     & & \twoheaddownarrow\\
     & & \cox(B_n)\\
  \end{array}$
  \caption{Middle groups and their relatives.\label{fig:middle-maps}}
\end{figure}

The maps in Figure~\ref{fig:middle-maps} are easy to describe.

\begin{defn}[Maps]\label{def:middle-maps}
  The map from $\midd(B_n)$ onto $\cox(B_n)$ can be seen
  geometrically.  The squares of the basic translations in
  $\mathcal{T}$ generate a normal subgroup $K \cong (2\Z)^n$ in
  $\midd(B_n)$ and if we quotient $\R^n$ by the action of $K$ the
  result is the $n$-torus $T^n$.  The kernel of the induced action of
  $\midd(B_n)$ on $T^n$ is $K$ and the action itself is easily seen to
  be the toroidal action of $\cox(B_n)$ on $T^n$.  To understand the
  horizontal map from $\cox(\wt A_{n-1})$ to $\midd(B_n)$ we note that
  the reflections in $\midd(B_n)$ acting on the hyperplane in $\R^n$
  perpendicular to the vector $\mathbf{1} = \langle 1^n \rangle$ is
  the standard realization of the euclidean symmetric group $\cox(\wt
  A_{n-1})$.  Its image in $\midd(B_n)$ is normal and the quotient
  sends $u \in \midd(B_n)$ to the sum of the coordinates of the image
  of the origin under $u$. We call the map from $\midd(B_n) \onto \Z$
  the \emph{vertical displacement map}.  The map from $\art(B_n)$ to
  $\midd(B_n)$ is clear when these are viewed as the annular braid
  group and annular symmetric group and the horizontal maps along the
  top row are well-known \cite{ChPe03}.  More precisely, the map from
  $\art(B_n)$ to $\Z$ sends elements to the sum of the winding numbers
  of the various paths from the disk to itself and the kernel of this
  map, the set of annular braids with global winding number $0$ is the
  group $\art(\wt A_{n-1})$.  We call the map from $\art(B_n) \onto
  \Z$ the \emph{global winding number map}.
\end{defn}

The middle column of Figure~\ref{fig:middle-maps} can be understood
via presentations.

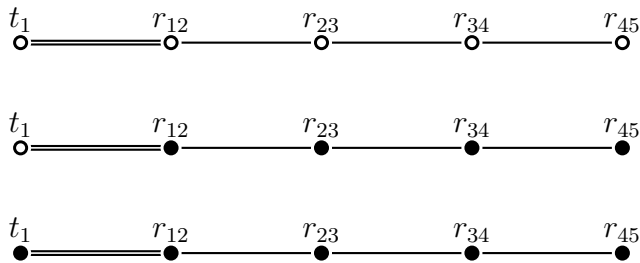
\begin{figure}
  \begin{center}
  \begin{tikzpicture}[scale=2]
    \foreach \y in {-.7,0,.7} {
      \draw[double,thick] (-2,\y)--(-1,\y);
      \draw[-,thick] (-1,\y)--(2,\y);
      \foreach \x in {-2,-1,0,1,2} {
        \fill[color=white] (\x,\y) circle (.7mm);
        \fill (\x,\y) circle (.5mm);}
      \draw (-2,\y) node [anchor=south] {$t_1$};
      \draw (-1,\y) node [anchor=south] {$r_{12}$};
      \draw (0,\y) node [anchor=south] {$r_{23}$};
      \draw (1,\y) node [anchor=south] {$r_{34}$};
      \draw (2,\y) node [anchor=south] {$r_{45}$};
    }
    \foreach \x in {-2,-1,0,1,2} \fill[color=white] (\x,.7) circle (.3mm);
    \fill[color=white] (-2,0) circle (.3mm);
  \end{tikzpicture}
  \end{center}
  \caption{Dynkin-diagram style presentations for the three groups
    $\art(B_5)$, $\midd(B_5)$ and $\cox(B_5)$.  In these diagrams
    solid circles indicate generators of order $2$ and empty circles
    indicate generators of infinite order. \label{fig:middle-dynkin}}
\end{figure}

\begin{defn}[Presentation]\label{def:middle-pres}
  A standard minimal generating set for $\midd(B_n)$ consists of
  adjacent transpositions $S = \{r_{ij} \mid j = i+1 \} \subset
  \mathcal{R}$ and the single translation $t_1$.  The set $S$
  generates all coordinate permutations and the other basic
  translations can be obtained by conjugating $t_1$ by a permutation.
  There are a number of obvious relations among these generators in
  addition to the standard Coxeter presentation for the symmetric
  group.  For example, $t_1$ commutes with $r_{ij}$ for $i,j >1$ and
  $t_1 r_{12} t_1 r_{12} = r_{12} t_1 r_{12} t_1$ since both motions
  are translations which add $1$ to the first two coordinates leaving
  the others unchanged.  These relations can be summarized in a
  diagram following the usual conventions: generators label the
  vertices, vertices not connected by an edge indicate generators
  which commute, vertices connect by a single edge indicate generators
  $a$ and $b$ which ``braid'' (i.e. $aba = bab$) and vertices
  connected by a double edge indicate generators $a$ and $b$ which
  satisfy the relation $abab = baba$.  For Coxeter groups, the
  generators have order $2$ and for Artin groups they have infinite
  order.  The middle groups are a mixed case: $t_1$ has infinite order
  but the adjacent transpositions have order~$2$.  See
  Figure~\ref{fig:middle-dynkin}.  It is an easy exercise to show that
  the relations encoded in the diagram for $\midd(B_n)$ are a
  presentation and as a consequence the surjections from $\art(B_n)$
  to $\midd(B_n)$ to $\cox(B_n)$ become clear.  Also note that the
  composition of these maps is the standard projection map from
  $\art(B_n)$ to $\cox(B_n)$.
\end{defn}

And finally we record a slightly more general context where groups
isomorphic to middle groups arise. These are the exact conditions
which occur in the later sections.

\begin{prop}[Recognizing middle groups]\label{prop:middle-recognize}
  Suppose the symmetric group $\sym_n$ acts faithfully by isometries
  on an $m$-dimensional euclidean space with root system $\Phi$ and $m
  \geq n$.  In addition, let $r$ be an element of a Coxeter generating
  set $S$ for $\sym_n$ representing one of two ends of the
  corresponding Dynkin diagram (so that $r$ does not commute with
  exactly one element of $S$).  If $t=t_\lambda$ is a translation such
  that $t$ does not commute with $r$, $t$ does commute with the rest
  of $S$, and $\lambda$ is not in the span of root system $\Phi$, then
  the group of isometries generated by $S \cup \{t\}$ is isomorphic to
  $\midd(B_n)$.
\end{prop}

\begin{proof}
  Let $G$ be the group generated by these elements.  First pick a
  point fixed by $\sym_n$ to serve as our origin and consider the
  $n$-dimensional subspace of $\R^m$ through this point spanned by the
  vectors $\Phi \cup \{\lambda\}$.  Because $S \cup \{t\}$ preserves
  this subspace and fixes its orthogonal complement, the same is true
  for the group $G$ that these isometries generate.  Thus we may
  restrict our attention to this subspace.  Next establish a
  coordinate system on this $\R^n$ so that $\sym_n$ is acting by
  coordinate permutations with the elements of $S$ switching adjacent
  coordinates and $r = r_{12}$.  From the conditions imposed on $t$ we
  know that in this coordinate system $\lambda = (a,b,b,\ldots,b)$
  with $a \neq b$.  Finally, note that the generators of $G$ (and thus
  every element of $G$) commute with the linear maps which fix the
  codimension one subspace spanned by $\Phi$ and rescale the vectors
  perpendicular to this subspace, i.e. the vectors with all
  coordinates equal.  After conjugating $G$ by the appropriate such
  map, we get the standard realization of $\midd(B_n)$ and because
  this conjugation is reversible the groups are isomorphic.
\end{proof}

\section{Intervals and noncrossing partitions}\label{sec:special-noncross}

Section~\ref{sec:perm-trans} discussed the groups $\cox(B_n)$,
$\midd(B_n)$ and $\art(B_n)$, the maps between them, and a
consistently labeled minimal generating set $\{t_1\} \cup S$ with $S =
\{r_{ij} \mid j = i+1\}$.  In this section we investigate intervals in
these groups and relate them to noncrossing partitions.

\begin{defn}[Special elements]\label{def:special-elts}
  Recall that a \emph{Coxeter element} is a element obtained by
  multiplying together the elements of some Coxeter generating set in
  a Coxeter group in some order and that for spherical Coxeter groups,
  or more generally for Coxeter groups whose Dynkin diagram is a tree,
  all Coxeter elements belong to a single conjugacy class.  For the
  group $\cox(B_n)$ we pick as our \emph{standard Coxeter element} $w$
  the product on the standard minimal generating set in the order they
  appear in the Dynkin diagram: $t_1$, $r_{12}$, $r_{23}$, and so on.
  Thus, for $\cox(B_5)$, shown in Figure~\ref{fig:middle-dynkin}, we
  have $w = t_1 r_{12} r_{23} r_{34} r_{45}$ and since we compose
  these as functions (from right to left) the element $w$ sends the
  point $(x_1,x_2,x_3,x_4,x_5)$ to the point $(-x_5, x_1, x_2, x_3,
  x_4)$.  The groups $\midd(B_n)$ and $\art(B_n)$ have analogues of the
  standard Coxeter element obtained by multiplying the corresponding
  generators together in the same fashion.  In $\midd(B_5)$ the
  resulting euclidean isometry $w = t_1 r_{12} r_{23} r_{34} r_{45}$
  sends the point $(x_1,x_2,x_3,x_4,x_5)$ to the point $(x_5+1, x_1,
  x_2, x_3, x_4)$ which is consistent with the reinterpretation of
  $t_1$ as a translation.  We call $w$ the \emph{special element} in
  all three contexts but in $\cox(B_n)$ it is more properly called a
  Coxeter element and in $\art(B_n)$ it is a \emph{dual Garside
    element}.
\end{defn}

\begin{figure}
  \begin{tikzpicture}[scale=3]
    \definecolor{lightblue}{rgb}{.85,.85,1};
    \definecolor{mediumblue}{rgb}{.7,.7,1};
    \fill[color=lightblue]
    (0:1cm)--(36:1cm)--(72:1cm)--(108:1cm)--(144:1cm)--(180:1cm)--
    (216:1cm)--(252:1cm)--(288:1cm)--(324:1cm)--cycle;
    \draw[dashed] (0:1cm)--(36:1cm)--(72:1cm)--(108:1cm)--(144:1cm)--(180:1cm)--
    (216:1cm)--(252:1cm)--(288:1cm)--(324:1cm)--cycle;

    \fill[color=mediumblue]  (36:1cm)--(72:1cm)--(324:1cm)--cycle;
    \fill[color=mediumblue]  (216:1cm)--(252:1cm)--(144:1cm)--cycle;
    \draw[thick] (36:1cm)--(72:1cm)--(324:1cm)--cycle;
    \draw[thick] (216:1cm)--(252:1cm)--(144:1cm)--cycle;
    \draw[thick] (108:1cm)--(288:1cm);

    \draw (0:1cm) node [anchor=west] {$-e_1$};
    \draw (36:1cm) node [anchor=south west] {$e_5$};
    \draw (72:1cm) node [anchor=south west] {$e_4$};
    \draw (108:1cm) node [anchor=south east] {$e_3$};
    \draw (144:1cm) node [anchor=south east] {$e_2$};
    \draw (180:1cm) node [anchor=east] {$e_1$};
    \draw (216:1cm) node [anchor=north east] {$-e_5$};
    \draw (252:1cm) node [anchor=north east] {$-e_4$};
    \draw (288:1cm) node [anchor=north west] {$-e_3$};
    \draw (324:1cm) node [anchor=north west] {$-e_2$};
    \foreach \x in {0,1,2,3,4,5,6,7,8,9} {\fill (36*\x:1cm) circle (.3mm);}
    \foreach \x in {0,1,2,3,4,5,6,7,8,9} {\fill[color=white] (36*\x:1cm) circle
      (.2mm);}
  \end{tikzpicture}
  \caption{A centrally symmetric noncrossing partition in a regular
    convex decagon.\label{fig:nc}}
\end{figure}
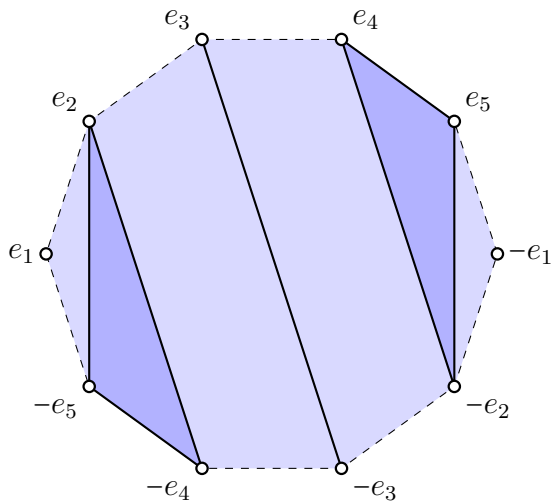

\begin{defn}[Noncrossing partitions]
  A \emph{noncrossing partition} is a partition of the vertices of a
  regular convex polygon so that the convex hulls of distinct blocks
  are disjoint.  A \emph{noncrossing partition of type $B$} is a
  noncrossing partition of an even-sided polygon whose blocks are
  symmetric with respect to a $\pi$-rotation about its center.
  Figure~\ref{fig:nc} shows a type $B$ noncrossing partition.  One
  partition is below another if every block of the first is contained
  in some block of the second.  Thus the partition where every block
  is a singleton is the minimum element and the partition with only
  one block is the maximum element.  
\end{defn}

It is well-known, at this point, that the type $B$ noncrossing
partitions correspond to the special interval in the type $B$ Coxeter
group \cite{Re97}.

\begin{lem}[Type $B$ intervals]
  If $W = \cox(B_n)$ is the type $B$ Coxeter group generated by all of
  its reflections and $w$ is its special element, then there is a
  natural identification of the interval $[1,w]^W$ and the type $B$
  noncrossing partitions of a $2n$-gon.
\end{lem}

\begin{proof}
  Every element in $W$ is determined by how it permutes the $2n$ unit
  vectors $\{ \pm e_i\}$ on the coordinate axes and under the element
  $w$ these form a single cycle of length $2n$: it sends $\pm e_i$ to
  $\pm e_{i+1}$ for $i<n$ and $\pm e_n$ to $\mp e_1$.  If we label the
  vertices of a $2n$-gon with the vectors in $\{ \pm e_i\}$ so that
  $w$ permutes them in a clockwise fashion (the case $n=5$ is shown in
  Figure~\ref{fig:nc}), then the interval $[1,w]^W$ is isomorphic as a
  poset to the type $B$ noncrossing partitions of this $2n$-gon.  The
  identification goes as follows: associate to each type $B$
  noncrossing partition the unique element of $W$ which sends the
  vector $e_i$ to the vector which occurs next in clockwise order in
  the boundary of the block to which $e_i$ belongs.  For example, the
  element $u$ corresponding to the partition shown in
  Figure~\ref{fig:nc} sends $e_1$ to $e_1$, $e_2$ to $-e_4$, $e_3$ to
  $-e_3$ and $e_4$ to $e_5$ and $e_5$ to $-e_2$.  Conversely, $u\in W$
  lies in the interval $[1,w]^W$ if and only if the orbits of these
  vectors under $u$ form noncrossing blocks which $u$ rotates in a
  clockwise manner.
\end{proof}
  
The following is thus only a minor extension of known results.

\begin{thm}[Special intervals]\label{thm:special-ints}
  Let $W = \cox(B_n)$, $M = \midd(B_n)$ and $A = \art(B_n)$ be the
  Coxeter group of type $B$, the middle group, and the Artin group of
  type $B$ with their standard full generating sets and let $w$ denote
  the special element in all three contexts.  The intervals $[1,w]^W$,
  $[1,w]^M$ and $[1,w]^A$ are isomorphic as labeled posets and their
  common underlying poset structure is that of the type $B$
  noncrossing partition lattice.  As a consequence, the group obtained
  by pulling a middle group apart at its special element is an annular
  braid group.
\end{thm}

\begin{proof}
  It is well known that the intervals $[1,w]^W$ and $[1,w]^A$ are
  isomorphic as labeled posets, that their common underlying poset is
  the type $B$ noncrossing partition lattice and that $A$ is the group
  whose presentation is encoded in $[1,w]^W$.  This is essentially
  what is meant when we say that spherical Artin groups have dual
  Garside presentations.  In particular, the final assertion is
  immediate once we show that $[1,w]^M$ is isomorphic to the others as
  a labeled poset.  Showing that all the factorizations in $[1,w]^W$
  lift from a factorization of an isometry of the $n$-torus $T^n$ to a
  factorization of the corresponding isometry of $\R^n$ and that no
  new factorizations arise is an easy exercise.
\end{proof}

These noncrossing diagrams make it easy to show that very few simples
commute with the special element.

\begin{prop}[Commuting with $w$]\label{prop:commuting-with-w}
  Let $G$ be the group $\cox(B_n)$, $\midd(B_n)$ or $\art(B_n)$ with
  its standard generating set and let $w$ be its special element.  The
  only elements in $[1,w]^G$ that commute with $w$ are the bounding
  elements $1$ and $w$.
\end{prop}

\begin{proof}
  If an element $u$ in $[1,w]^G$ commutes with $w$ then it must
  correspond to a centrally symmetric noncrossing partition of a
  $2n$-gon that is invariant under a $\frac{\pi}{n}$-rotation since
  this is how conjugation by $w$ acts on the type $B$ noncrossing
  partitions.  The only noncrossing partitions left invariant under
  this action are the partition in which every vertex belongs to a
  distinct block and the partition in which all the vertices belong to
  single block, and these correspond to $1$ and $w$ respectively.
\end{proof}

\begin{rem}[Generators and relations]
  The generators and relations visible inside $[1,w]^M$ can be given
  more explicitly.  The edge labels in the interval $[1,w]^W$ are
  exactly the $n^2$ reflection generators of $W = \cox(B_n)$ but this
  finite set is far from the full (infinite) generating set of $M =
  \midd(B_n)$.  In fact, the only elements which appear are
  $\mathcal{T} = \{t_i\}$, $\mathcal{R} = \{r_{ij}\}$ and the set
  $\mathcal{R}(1) = \{ r_{ij}(1)\}$.  The element $t_i$ corresponds to
  the diagonal edge connecting $e_i$ and $-e_i$, the element $r_{ij}$
  corresponds to the pair of edges connecting $\pm e_i$ to $\pm e_j$
  and the element $r_{ij}(1)$ corresponds to the pair of edges
  connecting $\pm e_i$ to $\mp e_j$.  If two generators correspond to
  edges which are completely disjoint, then there exists a commutation
  relation visible in the interval $[1,w]^M$.  For example, the length
  four cycle $t_1 r_{23} = r_{23} t_1$ can be found inside $[1,w]^M$.
  If two generators correspond to pairs of nondiagonal edges with only
  one endpoint in common, then dual braid relations are visible inside
  $[1,w]^M$ of the following form: there are three generators $a$, $b$
  and $c$ with $ab = bc = ca$ visible in the interval.  To illustrate,
  the generators $r_{45}$ and $r_{25}(1)$ share an endpoint and we
  have relations $r_{45} r_{25}(1) = r_{25}(1) r_{24}(1) = r_{24}(1)
  r_{45}$ and so $r_{45}$ and $r_{25}(1)$ braid in the corresponding
  interval group.  Finally, the generators $t_i$, $r_{ij}$, $t_j$ and
  $r_{ij}(1)$ (with $i<j$) satisfy a dual Artin relation of length
  $4$: $t_i r_{ij} = r_{ij} t_j = t_j r_{ij}(1) = r_{ij}(1) t_i$.
  These relations, taken together, are a complete presentation for the
  spherical Artin group $\art(B_n)$.
\end{rem}

Our final result in this section gives a new perspective on the
horizontal maps between the first two columns of
Figure~\ref{fig:middle-maps}.

\begin{prop}[Horizontal maps]\label{prop:hor-maps}
  If $M = \midd(B_n)$ is a middle group with special element $w$, then
  $(1)$ the reflections labeling edges in $[1,w]^M$ generate a copy of
  $\cox(\wt A_{n-1})$ inside $M$, $(2)$ the group generated by these
  elements and subject only to the relations among them visible in
  $[1,w]^M$ is isomorphic to $\art(\wt A_{n-1})$, and $(3)$ the
  natural projection map from this group to $M$ factors through and
  injects into the annular braid group $\art(B_n)$.
\end{prop}

\begin{proof}
  The reflections labeling an edge in $[1,w]^M$ are $\{r_{ij}\} \cup
  \{r_{ij}(1)\}$ and the subset $\{r_{ij} \mid j = i+1\} \cup
  \{r_{1n}(1)\}$ is already sufficient to generate the $\cox(\wt
  A_{n-1})$ subgroup of $M$ since they bound a chamber in the $\wt
  A_{n-1}$ tiling of the hyperplane orthogonal to the vector
  $\mathbf{1} = \langle 1^n \rangle$.  Next, notice that these
  elements correspond to the boundary edges of the $2n$-gon for $w$.
  As such they never cross and either braid or commute depending on
  whether or not they have endpoints in common.  Thus the group
  defined by just these generators and relations is isomorphic to the
  group $\art(\wt A_{n-1})$.  It is now straightforward to check the
  other generators and relations are consistent with this
  identification and what we have described is the standard copy of
  $\art(\wt A_{n-1})$ inside $\art(B_n)$.
\end{proof}

\part{New Groups}\label{part:new}
In this part we introduce several new groups closely related to each
irreducible euclidean Coxeter group and its corresponding Artin group.

\section{Intervals in euclidean Coxeter groups}\label{sec:euc-intervals}

Let $W = \cox(\wt X_n)$ be an irreducible euclidean Coxeter group with
reflections $R$ and Coxeter element $w$.  The coarse structure of the
interval $[1,w]^W$ was determined in the earlier articles
\cite{BrMc-factor} and \cite{Mc-lattice} and in this section we recall
the revelant definitions and results.  The first article, by Noel
Brady and the first author characterized the set of all possible
minimum length factorizations of a fixed euclidean isometry into
arbitrary reflections and the second showed that Coxeter intervals in
irreducible euclidean Coxeter groups are subposets of the unrestricted
intervals analyzed in the first article.  We begin by recalling the
distinction between points and vectors.

\begin{defn}[Points and vectors]
  Let $V$ denote an $n$-dimensional real vector space with the
  standard positive definite inner product and let $E$ be the
  corresponding euclidean analogue where the location of the origin
  has been forgotten leaving only a simply transitive action of $V$ on
  $E$.  The elements of $V$ are called \emph{vectors} and the elements
  of $E$ are called \emph{points}.  Ordered pairs of points in $E$
  determine a vector in~$V$.
\end{defn}

In \cite{BrMc-factor} euclidean isometries are analyzed in terms of
their two basic invariants: min-sets in $E$ and move-sets in $V$.

\begin{defn}[Basic invariants]
  Let $u$ be an isometry of $E$.  If $\lambda$ is the vector from $x$
  to $u(x)$ then we say that $x$ is \emph{moved by $\lambda$ under
    $u$}.  The collection $\mov(u) = \{ \lambda \mid x + \lambda =
  u(x)\} \subset V$ of all such vectors is the \emph{move-set} of $u$.
  The subset $\mov(u)$ is an affine subspace of $V$ and for each
  $\lambda \in \mov(u)$ the points of $E$ moved by $\lambda$ form an
  affine subspace of $E$ \cite[\BrMcInvAff]{BrMc-factor}.  In
  particular, there is a unique vector $\mu$ in $\mov(u)$ of minimal
  length and the corresponding points in $E$ form the \emph{min-set}
  of $u$, $\ms(u)$.  An isometry $u$ is \emph{elliptic} under the
  equivalent conditions that the vector $\mu$ is trivial, $\mov(u)$
  contains the origin in $V$ and there are points fixed by $u$.  For
  elliptic isometries we sometimes write $\fix(u)$ instead of
  $\ms(u)$.  Isometries that are not elliptic are called
  \emph{hyperbolic}.
\end{defn}

Let $L=\isom(E)$ be the Lie group of all euclidean isometries.  The
main results in \cite{BrMc-factor} analyze the structure of the
intervals in $\isom(E)$ with all reflections as its (trivially
weighted) generating set.  We call the interval $[1,w]^L$ an
\emph{ellptic} or \emph{hyperbolic interval} depending on the nature
of $w$.  In both cases, the elements of the intervals and the ordering
can be precisely described in terms of their basic invariants.  See
\cite{BrMc-factor} for details.  In this article we only need the
coarse structure of hyperbolic intervals where $w$ has maximal
reflection length, and in this context we define horizontal and
vertical directions.

\begin{defn}[Horizontal and vertical]
  If $w$ is a hyperbolic isometry whose min-set is a line, then the
  direction this line is translated is declared to be \emph{vertical}
  and the orthogonal directions are \emph{horizontal}.  A reflection,
  or more generally an elliptic isometry is called \emph{horizontal}
  if every point moves in a horizontal direction and it is vertical
  otherwise.  Thus a vertical elliptic isometry merely needs to have
  some vertical component to the motion of some point.
\end{defn}

A Coxeter element for the $\wt G_2$ tiling is a glide reflection and
thus an isometry of this type (see Figure~\ref{fig:g2-axis}).  Of the
$6$ families of parallel reflections there is one family of horizontal
reflections and five families of vertical reflections.  These can be
distinquished by whether or not their fixed hyperplanes cross the
glide axis.

\begin{figure}
  \begin{tikzpicture}[node distance=16mm, auto]
    \begin{scope}
      \tikzstyle{every node}=[rounded corners,draw]
      \node(00){$1$}; 
      \node[right of=00](10){$R_H$}; 
      \node[right of=10](20){}; 
      \node[right of=20](30){}; 

      \node[above of=00](01){$R_V$}; 
      \node[right of=01](11){}; 
      \node[right of=11](21){}; 
      \node[right of=21](31){}; 

      \node[above of=01](02){$T$}; 
      \node[right of=02](12){}; 
      \node[right of=12](22){}; 
      \node[right of=22](32){$w$}; 
    \end{scope}

    \node[right of=30](r1) {(ell,hyp) row};
    \node[right of=31](r2) {(ell,ell) row};
    \node[right of=32](r3) {(hyp,ell) row};

    \draw[-](00)--(10); \draw[-](20)--(30); 
    \draw[-](01)--(11); \draw[-](21)--(31); 
    \draw[-](02)--(12); \draw[-](22)--(32); 

    \fill ($(10)!.4!(20)$) circle (.4mm);
    \fill ($(10)!.5!(20)$) circle (.4mm);
    \fill ($(10)!.6!(20)$) circle (.4mm);
    \fill ($(11)!.4!(21)$) circle (.4mm);
    \fill ($(11)!.5!(21)$) circle (.4mm);
    \fill ($(11)!.6!(21)$) circle (.4mm);
    \fill ($(12)!.4!(22)$) circle (.4mm);
    \fill ($(12)!.5!(22)$) circle (.4mm);
    \fill ($(12)!.6!(22)$) circle (.4mm);

    \draw[-](00)--(01); \draw[-](01)--(02); 
    \draw[-](10)--(11); \draw[-](11)--(12);
    \draw[-](20)--(21); \draw[-](21)--(22); 
    \draw[-](30)--(31); \draw[-](31)--(32); 
  \end{tikzpicture}
  \caption{Coarse structure of a hyperbolic interval.\label{fig:coarse}}
\end{figure}
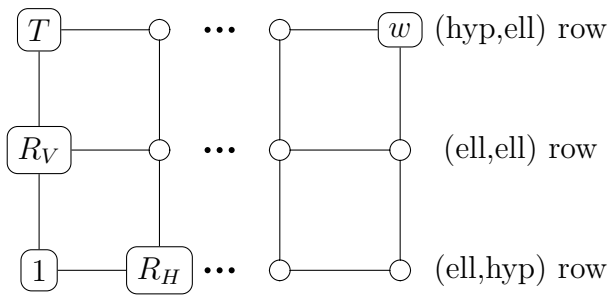

\begin{defn}[Coarse structure]
  Let $L = \isom(E)$ be the Lie group of all euclidean isometries and
  let $w$ be a hyperbolic euclidean isometry whose min-set is a line.
  For each element $u$ in the interval $[1,w]^L$ we consider the pair
  $(u,v)$ where $uv = w$.  There are exactly three possible cases: (1)
  $u$ is a horizontal elliptic isometry and $v$ is hyperbolic, (2)
  both $u$ and $v$ are vertical elliptic isometries, and (3) $u$ is
  hyperbolic and $v$ is horizontal elliptic.  These form the three
  rows of the \emph{coarse structure} of the interval arranged from bottom to
  top and shown in Figure~\ref{fig:coarse}.  The bottom row is graded
  by the dimension of the fixed set of $u$ from the identity element
  on the left to the elliptics fixing only a vertical line on the
  right.  The middle row has a similar grading: from those that fix a
  non-vertically invariant hyperplane on the left to those fixing only
  a single point on the right.  Alternatively, we could focus on $v$
  instead of $u$.  The $v$ on the left end of the middle row fix only
  a point and the $v$ on the right fix a non-vertically invariant
  hyperplane.  Finally, the top row is also graded by the fixed set of
  $v$: from $v$ fixing a vertical line on the left to $v$ equal to the
  identity on the right.  For every affine subspace of $E$ there is
  exactly one elliptic $u$ in one of the bottom two rows whose fix-set
  is this subspace.  Similarly, there is exactly one elliptic $v$ in
  one of the top two rows whose fix-set is this subspace.  Covering
  relations correspond to one horizontal or one vertical step in this
  grid.  Elements higher in the poset order are above and/or to the
  right while those lower down are down and/or to the left.  Finally,
  note that the second box on the bottom row contains the horizontal
  reflections, the first box in the middle row contains the vertical
  reflections, and the first box on the top row contains pure
  translations.
\end{defn}

\begin{figure}
  \includegraphics[scale=1]{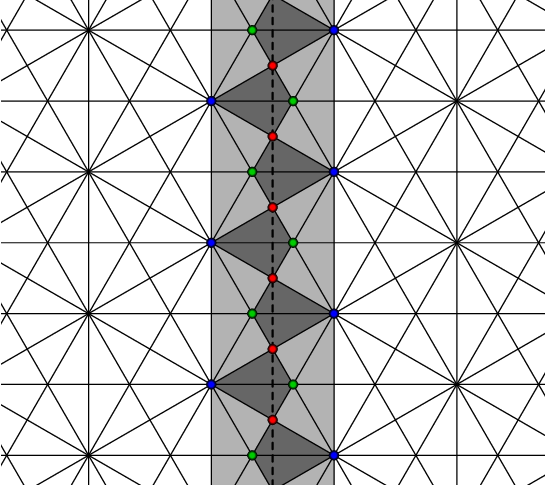}
  \caption{The $\wt G_2$ tiling of the plane with annotations
    corresponding to a particular Coxeter element
    $w$. \label{fig:g2-axis}}
\end{figure}

It turns out that for any irreducible euclidean Coxeter group $W =
\cox(\wt X_n)$ with reflections $R$ and Coxeter element $w$, the
\emph{Coxeter interval} $[1,w]^W$ is a subposet of the corresponding
hyperbolic interval in the full euclidean isometry group
\cite{Mc-lattice}.  In particular, it has the same basic structure.

\begin{defn}[Coxeter intervals]
  As described in \cite{Mc-lattice}, the min-set of the Coxeter
  element $w$ is a line $\ell$ called the \emph{Coxeter axis}.  Every
  point on this line is contained in the interior of some
  top-dimensional simplex, except for a discrete set of equally spaced
  points $x_i$ for $i \in \Z$.  The simplices through which $\ell$
  passes are called \emph{axial simplices} and the vertices of these
  simplices are \emph{axial vertices}.  The reflections which occur as
  edge labels in the interval $[1,w]^W$ are precisely those that
  contain an axial vertex in its fixed hyperplane
  \cite[\McReflections]{Mc-lattice}.  This includes all of the
  vertical reflections in $W$ but only a finite number of the
  horizontal ones.  We call these sets $R_V$ and $R_H$ respectively.
  Since the Coxeter axis passes through the interior of
  top-dimensional simplices, it does not lie on the hyperplane of any
  horizontal reflection.  For each family of parallel horizontal
  reflections, the only ones in the interval are the ones determined
  by the adjacent pair of hyperplanes which contain the Coxeter axis
  between them.  In other words, there are precisely two horizontal
  reflections in the interval for each antipodal pair of horizontal
  roots in the root system.
\end{defn}

The next lemma records a slightly technical fact about roots and axial
vertices that generalizes the observation above about horizontal
reflections.  It was verified by computer for the sporadic types and
by hand for the infinite families.

\begin{lem}[Convexity]\label{lem:convexity}
  Let $W = \cox(\wt X_n)$ be an irreducible euclidean Coxeter group
  with Coxeter element $w$ and let $r$ be a reflection that contains
  at least one axial vertex in its fixed hyperplane $H$.  If $\alpha$
  is a root in the type $X_n$ root system such that $\alpha$ has a
  positive dot product with the direction of the Coxeter axis and the
  image of $\alpha$ under the reflection $r$ has a negative dot
  product with the direction of the Coxeter axis, then the convex hull
  of the axial vertices contained in $H$ lies between two consecutive
  hyperplanes in the Coxeter complex with normal vector $\alpha$.
\end{lem}

Heuristically, the reason why Lemma~\ref{lem:convexity} is true is
that there are Coxeter elements whose axial vertices overlap with this
set of axial vertices in the hyperplane $H$ and in this alternative
world, the consecutive reflections with normal vector $\alpha$ are
horizontal with respect to the other Coxeter element and bound its
column of axial vertices.

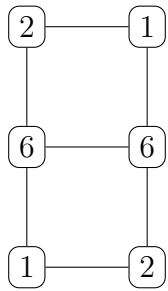
\begin{figure}
  \begin{tikzpicture}[node distance=16mm, auto]
    \tikzstyle{every node}=[rounded corners,draw]
    \node(00){$1$}; 
    \node[right of=00](10){$2$}; 

    \node[above of=00](01){$6$}; 
    \node[right of=01](11){$6$}; 

    \node[above of=01](02){$2$}; 
    \node[right of=02](12){$1$}; 

    \draw[-](00)--(10); 
    \draw[-](01)--(11); 
    \draw[-](02)--(12); 

    \draw[-](00)--(01); \draw[-](01)--(02); 
    \draw[-](10)--(11); \draw[-](11)--(12);
  \end{tikzpicture}
  \caption{Coarse structure of the $\wt G_2$ interval.\label{fig:g2-coarse}}
\end{figure}

\begin{exmp}[$\wt G_2$ interval]
  The $\wt G_2$ tiling of the plane is shown in
  Figure~\ref{fig:g2-axis} with various aspects highlighted.  The
  Coxeter element $w$ is a glide reflection whose glide axis is its
  min-set.  This is shown as a dashed line.  The heavily shaded
  triangles are the axial simplices of $w$ and the large dots indicate
  the axial vertices.  The lightly shaded vertical strip is the convex
  hull of the axial vertices and it is bounded by the only two
  horizontal reflections which occur in the Coxeter interval.  The
  coarse structure of the interval is shown in
  Figure~\ref{fig:g2-coarse}.  The numbers along the top and bottom
  rows represent the finite number of elements of each type in the
  interval.  Thus, $R_H$ contain two horizontal reflections and $T$
  contains two pure translations.  The middle row requires a more
  detailed explanation.  The convex hull has a structure which repeats
  vertically and the numbers in the middle row record how many
  distinct local situations there are in each box.  For example, there
  are infintely many vertical reflections in the interval but only six
  different types and there are infinitely many elliptic isometries in
  the interval that fix a single point but only six different types.
  In the former case the reflections are mostly distinguished by their
  slope but there are two with horizontal fixed lines that have
  distinct local neighborhoods. Similarly, in the latter case the
  rotations are mostly distinguished by the horizontal displacement of
  their fixed point except that there are two distinct types of fixed
  points along the Coxeter axis itself.  Both of these are
  $\pi$-rotations about their fixed point but they have distinct local
  neighborhoods and thus decompose into distinct types of reflections.
\end{exmp}

The coarse structure of the Coxeter interval in the largest of the
sporadic euclidean Coxeter groups offers a more substantial
illustration.

\begin{exmp}[$\wt E_8$ interval]
  The coarse structure of the Coxeter interval $[1,w]^W$ for the group
  $W=\cox(\wt E_8)$ is shown in Figure~\ref{fig:e8-coarse}.  From the
  figure we see that it contains $28$ horizontal reflections, $30$
  pure translations and $270$ infinite families of similarly situated
  vertical reflections.  In general, the numbers along the top and
  bottom refer to the number of individual elements in that box and
  the numbers in the middle row refer to number of infinite families
  of similarly situated elliptic elements.  We should note
  representatives of the roughly quarter-million types summarized in
  the figure were computed by a program \texttt{euclid.sage} written
  by the first author and available upon request.
\end{exmp}

\begin{figure}
  \begin{tikzpicture}[node distance=16mm, auto]
    \tikzstyle{every node}=[rounded corners,draw]
    \node(00){$1$}; 
    \node[right of=00](10){$28$}; 
    \node[right of=10](20){$235$}; 
    \node[right of=20](30){$826$}; 
    \node[right of=30](40){$1345$}; 
    \node[right of=40](50){$1000$}; 
    \node[right of=50](60){$315$}; 
    \node[right of=60](70){$30$}; 

    \node[above of=00](01){$270$}; 
    \node[right of=01](11){$5550$}; 
    \node[right of=11](21){$32550$}; 
    \node[right of=21](31){$75030$}; 
    \node[right of=31](41){$75030$}; 
    \node[right of=41](51){$32550$}; 
    \node[right of=51](61){$5550$}; 
    \node[right of=61](71){$270$}; 

    \node[above of=01](02){$30$}; 
    \node[right of=02](12){$315$}; 
    \node[right of=12](22){$1000$}; 
    \node[right of=22](32){$1345$}; 
    \node[right of=32](42){$826$}; 
    \node[right of=42](52){$235$}; 
    \node[right of=52](62){$28$}; 
    \node[right of=62](72){$1$}; 

    \draw[-](00)--(10); \draw[-](10)--(20); \draw[-](20)--(30); \draw[-](30)--(40); 
    \draw[-](40)--(50); \draw[-](50)--(60); \draw[-](60)--(70); 

    \draw[-](01)--(11); \draw[-](11)--(21); \draw[-](21)--(31); \draw[-](31)--(41); 
    \draw[-](41)--(51); \draw[-](51)--(61); \draw[-](61)--(71); 

    \draw[-](02)--(12); \draw[-](12)--(22); \draw[-](22)--(32); \draw[-](32)--(42); 
    \draw[-](42)--(52); \draw[-](52)--(62); \draw[-](62)--(72); 

    \draw[-](00)--(01); \draw[-](01)--(02); 
    \draw[-](10)--(11); \draw[-](11)--(12);
    \draw[-](20)--(21); \draw[-](21)--(22);
    \draw[-](30)--(31); \draw[-](31)--(32);
    \draw[-](40)--(41); \draw[-](41)--(42);
    \draw[-](50)--(51); \draw[-](51)--(52);
    \draw[-](60)--(61); \draw[-](61)--(62);
    \draw[-](70)--(71); \draw[-](71)--(72);
  \end{tikzpicture}
  \caption{Coarse structure of the $\wt E_8$ interval.\label{fig:e8-coarse}}
\end{figure}
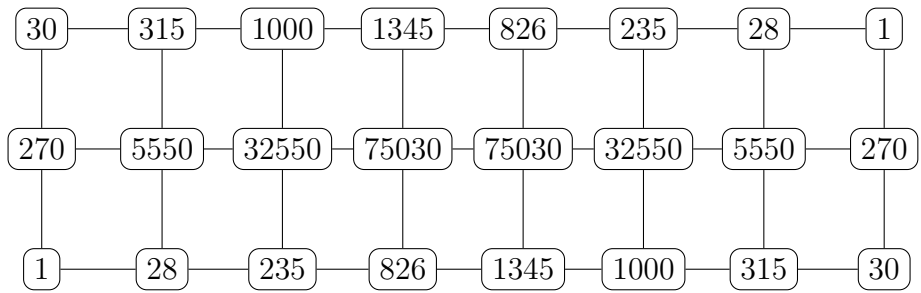

We conclude this section by reviewing an explicit presentation for the
dual euclidean Artin group derived from the Hurwitz action of the
braid group on factorizations of $w$.

\begin{defn}[Hurwitz action]
  Because reflections in $W$ are closed under conjugation,
  factorizations in $[1,w]^W$ can be rewritten in many ways and, in
  fact, there is an action of the braid group on the minimal length
  factorizations of $w$ called the \emph{Hurwitz action}.  The $i$-th
  standard braid generator replaces the two letter subword $ab$ in
  positions $i$ and $i+1$ with the subword $ca$ where $c = aba^{-1}$
  and it leaves the letters in the other positions unchanged.  It is
  easy to check that this action satisfies the relations in the
  standard presentation of the braid group.
\end{defn}

When a standard braid generator replaces $ab$ with $ca$ inside a
minimal length factorization of $w$, the relation $ab=ca$ is visible
in $[1,w]^W$.  Such a relation is called a \emph{Hurwitz relation} or
a \emph{dual braid relation}.  When the Hurwitz action is transitive
on factorizations, these relations are sufficient to define the
interval group $W_w$ \cite[\BrMcInvAff]{Mc-lattice} and we call this
the \emph{Hurwitz presentation}.  In 2010 Igusa and Schiffler proved
transitivity of the Hurwitz action on reflection factorizations of
Coxeter elements in Coxeter groups in complete generality
\cite{IgusaSchiffler10} and in 2014 a short proof of this general fact
was posted by Baumeister, Dyer, Stump and Wegener \cite{BDSW-hurwitz}.
As an illustration, we give the Hurwitz presentation of the dual $\wt
G_2$ Artin group.  We start with the generators.  The dual generators
are closely connected to the Coxeter axis of $w$ and we introduce a
notation that reflects this fact.

\begin{defn}[Dual $\wt G_2$ generators]\label{def:g2-gens}
  In the case of $\wt G_2$ we use the letters $a$ through $f$ to
  indicate the slope of its fixed line in the ascending order:
  $-\sqrt{3}$, $\frac{-1}{\sqrt{3}}$, $0$, $\frac{1}{\sqrt{3}}$,
  $\sqrt{3}$ and $\infty$, respectively.  See
  Figure~\ref{fig:g2-axis}.  Next, recall that the hyperplanes of the
  vertical reflections intersect the axis in an equally spaced set of
  points $x_i$ for $i\in \Z$ \cite[Section~8]{Mc-lattice}.  We use
  subscripts on the vertical reflections that indicates which $x_i$
  its hyperplane contains.  Note that not every combination of letter
  and subscript actually occurs.  For $\wt G_2$ we let $x_0$ be the
  intersection of one of the horizontal lines with the axis,
  specifically one which intersects an axial vertex on the lefthand
  side of the shaded vertical strip.  There are only two horizontal
  reflections in the interval $[1,w]^W$ and we call these $f_{-}$ and
  $f_{+}$.  Putting this all together, the dual generators of
  $\dart(\wt G_2,w)$ are the set $\{a_i,b_j,c_k,d_i,e_j,f_\ell\}$
  where $i = 1 \mod 4$, $j = 3 \mod 4$, $k = 0 \mod 2$ and $\ell \in
  \{+,-\}$.
\end{defn}

The periodicity of the subscripts corresponds to the fact that there
is a power of $w$ which acts as a pure translation in the direction of
the Coxeter axis.  In $\cox(\wt G_2)$ this power is $w^2$ and the
action of $w^2$ on the plane shifts the point $x_i$ to $x_{i+4}$.

\begin{defn}[Dual $\wt G_2$ relations]\label{def:dual-relations}
  The dual braid relations in the $\wt G_2$ case are obtained by
  factoring the elements in the interval $[1,w]^W$ of reflection
  length~$2$.  In the coarse structure, the elements to be factored
  belong to the third box in the bottom row, the second box in the
  middle row and the first box in the top row.  The first type does
  not occur in $\wt G_2$.  The third type are the pure translations
  and they have infinitely many factorizations.  In the case of $\wt
  G_2$ there are exactly two translations in the interval and their
  factorizations are as follows.
  \begin{equation}
    \begin{array}{c}
      \cdots = a_9 a_5 = a_5 a_1 = a_1 a_{-3} = a_{-3} a_{-7} = \cdots\\
      \cdots = e_{11} e_7 = e_7 e_3 = e_3 e_{-1} = e_{-1} e_{-5} =
      \cdots
    \end{array}
  \end{equation}
  It only remains to list the factorizations of the $6$ infinite
  families of elliptic elements that correspond to the second box in
  the middle row.  In the $\wt G_2$ case, these are rotations that fix
  a single point.  Representative sets of equations are as follows.
  \begin{equation}
    \begin{array}{c}
    a_1 d_1 = d_1 a_1\\
    b_3 e_3 = e_3 b_3\\
    c_2 a_1 = e_3 c_2 = a_1 e_3\\
    a_1 c_0 = e_{-1} a_1 = c_0 e_1\\
    a_{-3} f_{-} = b_{-1} a_{-3} = c_0 b_{-1} = d_1 c_0 = e_3 d_1 = f_{-} e_3 \\
    e_{-1} f_{+} = d_1 e_{-1} = c_2 d_1 = b_3 c_2 = a_5 b_3 = f_{+}
    a_5
    \end{array}
  \end{equation}
  To get all of the equations in the six infinite families, one should
  pick an arbitrary multiple of $4$ and consistently add it to each of
  the subscripts in each of six lines of equations above.  This
  corresonds to the vertical shift which conjugation by $w^2$
  produces. The $+/-$ subscripts remain underchanged since these
  reflections are invariant under vertical translation.
\end{defn}

\section{Horizontal roots and factored translations}\label{sec:horizontal}

In this section we describe the roots that are horizontal with respect
to the axis of a Coxeter element and we use their geometry to define a
series of crystallographic groups acting geometrically on euclidean
space.  Although Coxeter elements are usually defined as a product of
the reflections fixing the facets of a chamber in the Coxeter tiling,
there are other factorizations and one in particular where most of the
reflections are horizontal with respect to its axis.

\begin{defn}[Horizontal roots]
  If $w$ is a Coxeter element for the irreducible euclidean Coxeter
  group $W = \cox(\wt X_n)$, then $w$ has a factorization into a pure
  translation and $n-1$ horizontal reflections.  To see this we start
  with a standard factorization such as $w = r_{\alpha,1} w_0$ where
  $r_{\alpha,1}$ is the reflection corresponding to the root $\alpha$
  used to extend the Dynkin diagram $X_n$ shifted so it does not fix
  the origin and $w_0$ is a Coxeter element of the spherical Coxeter
  group $W_0 = \cox(X_n)$.  By Proposition~\ref{prop:sph-cox-elt} we
  can find an alternative factorization of $w_0$ as the product of a
  Coxeter generating set whose leftmost reflection is $r_\alpha$.
  Thus we can write $w_0 = r_\alpha w_h$ where $w_h$ is a Coxeter
  element of a maximal parabolic subgroup of $W_0$.  This means that
  $w = r_{\alpha,1} r_\alpha w_h = t_{\alpha^\vee} w_h$.  Since the
  element $w_h$ is an elliptic isometry fixing a line and $t =
  t_{\alpha^\vee}$ is a pure translation, the fixed line of $w_h$ must
  be parallel to the Coxeter axis of $w$.  As a consequence the $n-1$
  reflections multiplied together to produce $w_h$ are horizontal with
  respect to the axis of $w$.  Moreover, since the fixed line of $w_h$
  passes through the fixed point of $w_0$ and every family of parallel
  hyperplanes contains one which passes through this fixed point,
  these $n-1$ horizontal reflections generate a group $W_h$ with one
  representative from every parallel family of horizontal reflections
  in $W$.  In other words, all reflections in $W_h$ are horizontal and
  every horizontal reflection is parallel to one in $W_h$.  We call
  $W_h$ the \emph{horizontal Coxeter group} and $w = t_{\alpha^\vee}
  w_h$ a \emph{horizontal factorization of $w$}.  The horizontal roots
  associated to these reflections are a root system described by the
  diagram for $W_h$, and this diagram is the diagram for $W_0$ with an
  additional vertex removed, the one shown in
  Figures~\ref{fig:dynkin-families} and~\ref{fig:dynkin-sporadic} as a
  large shaded dot.  We call the corresponding root the \emph{vertical
    root}.
\end{defn}

\begin{rem}[Finding vertical roots]  
  The vertical roots were first found in \cite{Mc-lattice} on a
  case-by-case basis but once the principles are clear they can be
  easily spotted.  Because the simple system for $W_0$ used to create
  the horizontal factorization spans a positive cone, the vertical
  root should be as close to horizontal as possible.  This favors
  branch points and vertices involved in multiple bonds, specifically
  the end corresponding to the longer root.  This rule uniquely
  determines the vertice root in all cases except in type $A$ where
  there are distinct conjugacy classes of Coxeter element that lead to
  distinct choices of vertical root.
\end{rem}

The next proposition records some basic facts about the pure
translations that occur in the interval $[1,w]^W$ of an irreducible
euclidean Coxeter group $W$.  These are easily checked by hand for the
infinite families and by computer for the sporadic types.

\begin{prop}[Pure translations below $w$]\label{prop:trans}
  If $W = \cox(\wt X_n)$ is an irreducible euclidean Coxeter group
  with Coxeter element $w$, then every pure translation $t$ contained
  in the interval $[1,w]^W$ is the translation part of some horizontal
  factorization of $w$.  Moreover, if $t = r' r$ is a factorization of
  $t$ into a pair of reflections, then $r' = (w^p) r (w^{-p})$ where $w^p$
  is the smallest power of $w$ which acts on the Coxeter complex as a
  pure translation.  In fact, all factorizations of $t$ in $[1,w]^W$
  are of the form $t = r_{i+1} r_i$ where $r_i = (w^{ip}) r (w^{-ip})$ for
  some integer~$i$.
\end{prop}

\begin{defn}[Components]
  The structure of the horizontal root system is listed in
  Table~\ref{tbl:horizontal} for each irreducible type and note that
  the number of irreducible components varies from one to three.  The
  groups of type $\wt C_n$, $\wt A_n$ (with $q=1$) and $\wt G_2$ have
  a single component, the groups of type $\wt B_n$, $\wt A_n$ (with $q
  \geq 2$) and $\wt F_4$ have two components, and the groups of type
  $\wt D_n$, $\wt E_6$, $\wt E_7$ and $\wt E_8$ have three components.
  We orthogonally decompose the space $V$ of vectors into components
  as follows: $V = V_0 \oplus \cdots \oplus V_k$ where $V_0$ is the
  line spanned by the direction of the Coxeter axis and the components
  $V_i$ for $1 \leq i \leq k$ correspond to the subspaces spanned by
  the irreducible components of the horizontal root system.  Since
  every horizontal reflection corresponds to a root in exactly one of
  component $V_i$, we can partition any minimal Coxeter generating set
  $S_H$ for $W_h$ and the full set of reflections $R_H$ into disjoint
  subsets $S_H^{(i)}$ and $R_H^{(i)}$.
\end{defn}

\begin{table}
  $\begin{array}{|c|l|}
    \hline
    \textrm{Type} & \textrm{Horizontal root system}\\
    \hline
    A_n & \Phi_{A_{p-1}} \cup \Phi_{A_{q-1}} \\
    C_n & \Phi_{A_{n-1}} \\
    B_n & \Phi_{A_1} \cup \Phi_{A_{n-2}} \\
    D_n & \Phi_{A_1} \cup \Phi_{A_1} \cup \Phi_{A_{n-3}} \\
    \hline
    G_2 & \Phi_{A_1} \\
    F_4 & \Phi_{A_1} \cup \Phi_{A_2} \\
    E_6 & \Phi_{A_1} \cup \Phi_{A_2} \cup \Phi_{A_2} \\
    E_7 & \Phi_{A_1} \cup \Phi_{A_2} \cup \Phi_{A_3} \\
    E_8 & \Phi_{A_1} \cup \Phi_{A_2} \cup \Phi_{A_4} \\
    \hline
    \end{array}$\medskip
    \caption{The structure of the horizontal root system for each
      irreducible euclidean Coxeter group $\cox(\wt X_n)$.  For the
      group of type $A_n$ we list the structure of system of roots
      horizontal with respect to the axis of the $(p,q)$-bigon Coxeter
      elements defined in \cite{Mc-lattice}.\label{tbl:horizontal}}
\end{table}

It was an early hope that every dual Artin group would be a Garside
group, but it was shown in \cite{Mc-lattice} that this is not always
the case, even when attention is restricted to Artin groups of
euclidean type.  It turns out that the number of components of the
horizontal root system is crucial.

\begin{rem}[Garside structures]
  In \cite{Mc-lattice} the first author proved that the unique dual
  presentation of $\art(\wt X_n)$ is a Garside structure when $X$ is
  $C$ or $G$ and it is not a Garside structure when $X$ is $B$, $D$,
  $E$ or $F$.  When the group has type $A$ there are distinct dual
  presentations and the one investigated by Digne is the only one that
  is a Garside structure.  The positive results for types $A$ and $C$
  are due to Digne.  The negative results are a direct consequence of
  horizontal root systems with more than one irreducible component.
  The reducibility leads directly to a failure of the lattice
  condition \cite[\McReducibleBowties]{Mc-lattice}.  Knowing
  explicitly how and why the lattice property fails led to the groups
  we introduce below.  The second author worked out the structure of
  the Artin group of type $\wt B_3$ along the lines presented here in
  his dissertation under the supervision of the first author and it is
  these arguments that have now been generalized to arbitrary Artin
  groups of euclidean type \cite{Sulway-diss}.
\end{rem}

\begin{defn}[Diagonal translations]\label{def:diagonal-trans}
  Let $w$ be a Coxeter element in an irreducible euclidean Coxeter
  group $W = \cox(\wt X_n)$ and let $w = t_\lambda w_h$ be a
  horizontal factorization of $w$.  We call the translation
  $t_\lambda$ a \emph{diagonal translation} because $\lambda$ projects
  nontrivially to each of the components $V_i$ $0 \leq i \leq k$.  The
  vector $\lambda$ projects nontrivially to $V_0$, the direction of
  the Coxeter axis, because $w$ translates the axis vertically but the
  element $w_h$ only moves points horizontally.  And $\lambda$
  projects nontrivially to each horizontal component $V_i$ with $i>0$
  because the vertical root is connected by an edge to each component
  of the horizontal root system in the diagram $X_n$.  Also note that
  $t_\lambda$ is not orthogonal to exactly one reflection in each
  horizontal component.
\end{defn}

\begin{defn}[Factored translations]\label{def:factored-trans}
  Let $w$ be Coxeter element in an irreducible euclidean Coxeter group
  $W = \cox(\wt X_n)$ with a fixed horizontal factorization of $w$ and
  let $k$ be the number of horizontal components.  Let $t_\lambda$ be
  the corresponding vertical root translation and let $\lambda_i =
  \proj_{V_i}(\lambda)$ denote the nontrivial projection vectors to
  each subspace $V_i$.  Finally let $t_i = t_{\lambda_i} + \frac{1}{k}
  t_{\lambda_0}$ so that $t = \prod_{i=1}^k t_i$.  The translations
  $t_i$ are called \emph{factored translations}.  If we do this for
  every translation in the interval $[1,w]^\ccox$ then we get a
  collection $T_F$ of all factored translations. Like the horizontal
  reflections, they can be partitioned into subsets $T_F^{(i)}$ based
  on the particular component of the horizontal root system involved
  so that $T_F^{(i)}$ contains the factored translations whose
  displacement vector lies in $V_0 \oplus V_i$.
\end{defn}  

\begin{table}
\begin{tabular}{|c|c|l|}
\hline
\textbf{Name} & \textbf{Symbol} & \textbf{Generating set}\\
\hline
Coxeter & $\ccox$ & $R_H \cup R_V\ (\cup\ T )$\\
Horizontal & $\chor$ & $R_H$\\
Diagonal & $\cdiag$ & $R_H \cup T$\\
Factorable & $\cfac$ & $R_H \cup T_F\ (\cup\ T)$\\
Crystallographic & $\ccryst\ $ & $R_H \cup R_V \cup T_F\ (\cup\ T)$\\
\hline
\end{tabular}\medskip
\caption{Five euclidean isometry groups.\label{tbl:five-gps}}
\end{table}

For each Coxeter element in an irreducible euclidean Coxeter group
there are five closely related euclidean isometry groups that are
involved in our proofs.

\begin{defn}[Five euclidean isometry groups]\label{def:five-euc-gps}
  Let $W = \cox(\wt X_n)$ be an irreducible euclidean Coxeter group.
  For each choice of Coxeter element $w$, we have defined four sets of
  euclidean isometries: the horizontal reflections $R_H$ and the
  vertical reflections $R_V$ labeling edges in the interval $[1,w]^W$,
  the translations $T$ and the factored translations $T_F$.  Various
  combinations of these sets generate five euclidean isometry groups
  as shown in Table~\ref{tbl:five-gps}.  The \emph{horizontal group}
  $\chor$ is the euclidean isometry group generated by the set $R_H$
  of horizontal reflections below $w$.  It contains but is bigger than
  the group $W_h$ because it contains two horizontal reflections for
  each horizontal root.  The \emph{diagonal group} $\cdiag$ is the
  euclidean isometry group generated by $R_H \cup T$, the horizontal
  reflections and the pure translations below $w$.  The
  \emph{factorable group} $\cfac$ is the euclidean isometry group
  generated by $R_H \cup T_F$.  And the \emph{crystallographic group}
  $\ccryst = \cryst(\wt X_n,w)$ is the group generated by the union of
  all four sets.  Since every diagonal translation can be written
  either as a product of two parallel vertical reflections or as a
  product of $k$ factored translations, the set $T$ can be optionally
  included in the generating sets for $\ccox$, $\cfac$ and $\ccryst$
  without altering the group.
\end{defn}

The crystallographic group $\ccryst =
\cryst(\wt X_n,w)$ and the Coxeter group $W=\cox(\wt X_n)$ have a very
similar structure.

\begin{rem}[Crystallographic]
  Recall that a group action on a metric space is \emph{geometric}
  when the group acts properly discontinuously and cocompactly by
  isometries and that a group acting geometrically on a finite
  dimensional euclidean space is a \emph{crystallographic group}.
  This category includes but is larger than the class of euclidean
  Coxeter groups since crystallographic groups do not need to be
  generated by reflections.  For example, most of the $17$ distinct
  wallpaper groups acting geometrically on the plane are not euclidean
  Coxeter groups.  The group $\ccryst=\cryst(\wt X_n,w)$ is
  crystallographic because its structure is essentially the same as
  that of the Coxeter group $W=\cox(\wt X_n,w)$.  In particular, it
  has a normal translation subgroup with quotient spherical Coxeter
  group $W_0=\cox(X_n)$.  The only difference is that the new
  translation subgroup is slightly bigger: the old translation
  subgroup is finite index in the new one.
\end{rem}

\section{Intervals in the new groups}\label{sec:new-groups}

In this section we define and analyze intervals in four of the groups
introduced in the previous section.  We begin by extending our system
of weights to the larger generating sets.

\begin{defn}[Weights]
  We extend the trivial weighting on the full set $R$ of reflections
  so that the new factorizations preserve length.  The natural weights
  assign $1$ to each horizontal and vertical reflection, $2$ to each
  diagonal translation and $\frac{2}{k}$ to each factored translation
  where $k$ is the number of components of the horizontal root system.
\end{defn}

\begin{figure}
  \begin{tikzpicture}[node distance=2cm, auto]
    \node(Hw){$\ghor$}; 
    \node[right of=Hw](Dw){$\gdiag$}; 
    \node[right of=Dw](Ww){$\gart$}; 
    \node(Fw)[right of=Dw, above of=Dw, node distance=1cm]{$\gfac$}; 
    \node[right of=Fw](Cw){$\ggar$};
    \node[below of=Fw](F){$\cfac$}; 
    \node[below of=Cw](C){$\ccryst$}; 
    \node[below of=Hw](H){$\chor$};
    \node[below of=Dw](D){$\cdiag$}; 
    \node[below of=Ww](W){$\ccox$};
    \draw[->](Fw)--(Cw); 
    \draw[->](Hw)--(Dw); 
    \draw[->](Dw)--(Ww); 
    \draw[->](Dw)--(Fw); 
    \draw[->](Ww)--(Cw);
    \draw[right hook->](F)--(C); 
    \draw[right hook->](H)--(D); 
    \draw[right hook->](D)--(W);
    \draw[right hook->](D)--(F); 
    \draw[right hook->](W)--(C);
    \draw[->>](Fw)--(F); 
    \draw[->>](Cw)--(C); 
    \draw[->>](Hw)--(H);
    \draw[->>](Dw)--(D); 
    \draw[->>](Ww)--(W);
  \end{tikzpicture}
  \caption{Ten groups defined for each choice of a Coxeter element in
    an irreducible euclidean Coxeter group and some of the maps
    between them.\label{fig:ten-gps}}
\end{figure}
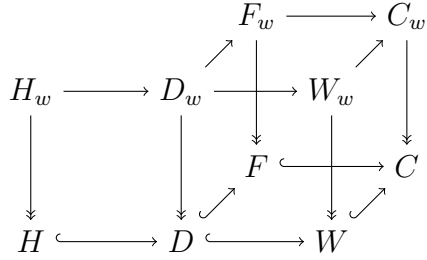

With this system of weights, the intervals behave as expected.  There
are inclusions among the intervals $[1,w]^X$ where $X$ is $D$, $F$,
$W$ or $C$, that mimic the relations between the groups as shown in
Figure~\ref{fig:ten-gps}.  The next lemma records additional relations
among these intervals.

\begin{lem}[Interval relations]\label{lem:int-rel}
  For each choice of a Coxeter element $w$ in an irreducible euclidean
  Coxeter group, the intervals described above are related as follows:
  \[[1,w]^{\ccryst} = [1,w]^{\ccox} \cup [1,w]^{\cfac}\]
  \[[1,w]^{\cdiag} = [1,w]^{\ccox} \cap [1,w]^{\cfac}\]
\end{lem}

\begin{proof}
  The second equality is an immediate consequence of the relations
  among the generating sets, as is the fact that $[1,w]^{\ccryst}
  \supset [1,w]^{\ccox} \cup [1,w]^{\cfac}$.  It only remains to show
  that there does not exist a minimal length factorization of $w$ in
  $\ccryst$ that includes both a factored translation and a vertical
  reflection.  To see this consider the map from $\ccryst$ to $W_0$
  obtained by quotienting out its normal subgroup of pure
  translations.  The image of $w$ under this map is a Coxeter element
  for the horizontal Coxeter group $W_h$.  It fixes a line parallel to
  the Coxeter axis through the unique point fixed by all of $W_0$.
  Since its move-set is $(n-1)$-dimensional, its minimal reflection
  length is $n-1$, and this length is only possible if each of the
  $n-1$ reflections in the product contain the fixed line in their
  fixed hyperplane.  In other words, this happens only when they are
  all horizontal reflections.  When this minimum is not achieved, at
  least $n+1$ reflections are involved because of parity issues.  If
  we start with a factorization of $w$ that contains a factored
  translation, then its image in $W_0$ has length strictly less than
  $n+1$, and as a consequence all of the reflections involved are
  horizontal.
\end{proof}

\begin{figure}
  \begin{tikzpicture}[scale=1.5]
    \tikzstyle{every node}=[rounded corners,draw]
    \node (T) at (0,1) {Top}; 
    \node (M) at (-1,0) {Middle}; 
    \node (F) at (1,0) {Factored}; 
    \node (B) at (0,-1) {Bottom}; 
    \draw[-](B)--(M)--(T)--(F)--(B);
  \end{tikzpicture}
  \caption{A very coarse overview of the structure of the interval
    $[1,w]^C$.\label{fig:cryst-coarse}}
\end{figure}
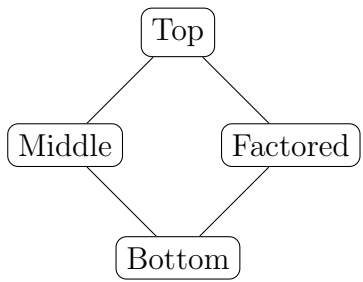

Using Lemma~\ref{lem:int-rel} we extend the notion of a coarse
structure to these new intervals.

\begin{rem}[Coarse structure]\label{rem:diag}
  The crystallographic interval $[1,w]^\ccryst$ is obtained by adding
  additional elements to the original three rows in the coarse
  structure of the Coxeter interval $[1,w]^\ccox$.  This is
  schematically shown in Figure~\ref{fig:cryst-coarse} but the reader
  should note that the box labeled Factored is not a single row but
  rather it includes all factorization pairs $(u,v)$ with $uv=w$ where
  both $u$ and $v$ require a factored translation in their
  construction.  The original Coxeter interval $[1,w]^\ccox$ is the
  subposet containing the top, middle and bottom portion, the diagonal
  interval $[1,w]^\cdiag$ is the poset containing only the top and
  bottom rows, and the factor interval $[1,w]^\cfac$ is the subposet
  containing only the top, bottom and factored portions.  One
  consequence of this is that the groups $\cdiag$ (and the pulled
  apart group $\gdiag$ defined below) have alternate generating sets.
  Instead of using $R_H \cup T$ we could instead use $R_H \cup \{w\}$.
  This is because every element in the bottom row is a product of
  horizontal reflections and every element in the top row differs from
  $w$ by a product of horizontal reflections.
\end{rem}

There are other properties that are nearly immediate.

\begin{prop}[Balanced and self-dual]\label{prop:balanced}
  For each choice of a Coxeter element $w$ in an irreducible euclidean
  Coxeter group, the interval between $1$ and $w$ in each of $\cdiag$,
  $\ccox$, $\cfac$ and $\ccryst$ is a balanced and self-dual poset.
\end{prop}

\begin{proof}
  Each interval is balanced because the generating sets are closed
  under local conjugations.  This also means that the map sending $u$
  to its left complement is an order-reversing poset isomorphism.
\end{proof}

Using these intervals we can create new groups.

\begin{defn}[Five groups via presentations]\label{def:five-pres-gps}
  Four of the groups on the top level of Figure~\ref{fig:ten-gps} are
  interval groups obtained by pulling apart the corresponding groups on
  bottom level.  The exception is $\ghor$.  We define this group as
  the group generated by the horizontal reflections $R_H$ in the
  interval $[1,w]^W$ and subject only to the relations among them that
  are visible there.  There is not a natural interval group here
  because $w$ itself is not an element of $H$; it is merely the
  horizontal portion of the other groups on the top level.  Finally,
  we should note that the groups $\ggar$ and $\gart$ turn out to be
  the Garside group described in the introduction and the Artin group
  $\art(\wt X_n)$ respectively.
\end{defn}

The inclusion relations among the various generating sets suffice to
establish the injections shown on the lower level of
Figure~\ref{fig:ten-gps} and inclusions among the sets of relations
induce the homomorphisms on the top level.  It turns out that all the
maps on the top level are also injective but this is not immediately
clear.  Several of these groups are easily identified.

\begin{prop}[Products]\label{prop:products}
  If $W = \cox(\wt X_n)$ is an irreducible euclidean Coxeter group
  with Coxeter element $w$ and $k$ horizontal components, then the
  interval $[1,w]^{\cfac}$ is a direct product of $k$ type~$B$
  noncrossing partition lattices and $\cfac$ is a central product of
  $k$ middle groups.  As a consequence:
  \begin{enumerate}
  \item $\gfac$ is a direct product of $k$ annular braid groups,
  \item $\ghor$ is a direct product of $k$ euclidean braid groups, and
  \item $\chor$ is a direct product of $k$ euclidean symmetric groups.
  \end{enumerate}
\end{prop}

\begin{proof}
  The group $F$ is minimally generated by the set $S_H \cup \{t_i\}$
  contained inside $R_H \cup T_F$ (with the $t_i$ being the factors of
  the diagonal translation $t_\lambda$ as described in
  Definition~\ref{def:factored-trans}) and note that both $S_H$ and
  $\{t_i\}$ can be partitioned based on the unique component of the
  horizontal root system involved in each motion.  By
  Proposition~\ref{prop:middle-recognize} the elements associated with
  each component generate a middle group.  Moreover, since generators
  associated to different components commute and $w$ can be factored
  into a product of special elements for these middle groups, the
  interval $[1,w]^F$ is a direct product of special intervals in
  middle groups.  By Theorem~\ref{thm:special-ints} each of these is a
  type~$B$ noncrossing partition lattice.  This also means that $F$ is
  almost, but not quite, a direct product of these middle groups
  because these groups have a nontrivial intersection.  They overlap
  in elements whose motions lie solely in the $V_0$ direction, a
  description which only applies to the pure translations that form
  their centers.  Thus $F$ is a central product rather than a direct
  product.  On the other hand, since $[1,w]^F$ is a direct product of
  lattices with disjoint edge labels, $\gfac$ is a direct product of
  annular braid groups.  The group $\ghor$ and $\chor$ are identified
  by applying Proposition~\ref{prop:hor-maps} to each factor.
\end{proof}

We illustrate Proposition~\ref{prop:products} with a concrete example.

\begin{exmp}[$\wt E_8$ groups]\label{ex:e8-cfac-chor}
  Since the horizontal $E_8$ root system decomposes as $\Phi_{A_1}
  \cup \Phi_{A_2} \cup \Phi_{A_4}$ (Table~\ref{tbl:horizontal}), the
  group $\cfac$ is a central product of $\midd(B_2)$, $\midd(B_3)$ and
  $\midd(B_5)$.  In addition,
  \begin{itemize}
    \item $[1,w]^F \cong NC_{B_2} \times NC_{B_3} \times NC_{B_5}$,
    \item $\gfac \cong \art(B_2) \times \art(B_3) \times \art(B_5)$,
    \item $\ghor \cong \art(\wt A_1) \times \art(\wt A_2) \times
      \art(\wt A_4)$, and
    \item $\chor \cong \cox(\wt A_1) \times \cox(\wt A_2) \times \cox(\wt A_4)$.
  \end{itemize}
\end{exmp}

\part{Main Theorems}\label{part:main-theorems}
In this final part we prove our four main results.

\section{Proof of Theorem A: Crystallographic Garside groups}

In this section we prove our first main result, that for every choice
of a Coxeter element $w$ in an irreducible euclidean Coxeter group $W
= \cox(\wt X_n)$, the group $\ggar = \gar(\wt X_n,w)$ is a Garside
group.  The most difficult step is to establish the lattice property
and we begin with a lemma which show that in discretely graded posets,
it is sufficient to work inductively and to establish that all pairs
of atoms have a well-defined join.

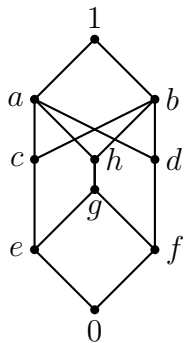
\begin{figure}
  \begin{tikzpicture}
    \coordinate (one) at (0,1.2);
    \coordinate (zero) at (0,-2.4);
    \coordinate (a) at (-.8,.4);
    \coordinate (b) at (.8,.4);
    \coordinate (c) at (-.8,-.4);
    \coordinate (d) at (.8,-.4);
    \coordinate (e) at (-.8,-1.6);
    \coordinate (f) at (.8,-1.6);
    \coordinate (g) at (0,-.8);
    \coordinate (h) at (0,-.4);
    \draw[-,thick] (b)--(one)--(a)--(e)--(zero)--(f);
    \draw[-,thick] (a)--(d)--(b)--(c);
    \draw[-,thick] (f)--(g)--(h)--(b);
    \draw[-,thick] (e)--(g)--(h)--(a);
    \draw[-,thick] (d)--(f);
    \fill (one) circle (.6mm) node[anchor=south] {$1$};
    \fill (a) circle (.6mm) node[anchor=east] {$a$};
    \fill (b) circle (.6mm) node[anchor=west] {$b$};
    \fill (c) circle (.6mm) node[anchor=east] {$c$};
    \fill (d) circle (.6mm) node[anchor=west] {$d$};
    \fill (e) circle (.6mm) node[anchor=east] {$e$};
    \fill (f) circle (.6mm) node[anchor=west] {$f$};
    \fill (g) circle (.6mm) node[anchor=north] {$g$};
    \fill (h) circle (.6mm) node[anchor=west] {$h$};
    \fill (zero) circle (.6mm) node[anchor=north] {$0$};
  \end{tikzpicture}
  \caption{Posets elements used in the proof of
    Lemma~\ref{lem:atom-subint}.\label{fig:atom-subint}}
\end{figure}

\begin{lem}[Atoms and subintervals]\label{lem:atom-subint}
  Let $P$ be a bounded poset that is graded with respect to a discrete
  weighting.  If all pairs of atoms in $P$ have well-defined joins and
  $P$ is not a lattice, then $P$ contains a proper subinterval that is
  not a lattice.
\end{lem}

\begin{proof}
  Since $P$ is not a lattice, it contains a bowtie $(a,b:c,d)$ by
  Proposition~\ref{prop:bowtie}.  Let $e$ and $f$ be atoms in $P$
  below $c$ and $d$ respectively.  By assumption atoms $e$ and $f$
  have a join $g = e \join f$ and since $a$ and $b$ are upper bounds
  for $e$ and $f$, we have $a \geq g$ and $b \geq g$ by definition of
  being a join.  Finally, let $h$ a maximal lower bound for $a$ and
  $b$ that is above $g$.  See Figure~\ref{fig:atom-subint} and note
  that such $e$, $f$ and $h$ exist because of the discreteness of the
  grading.  If $h \neq c$, then $(a,b:c,h)$ is a bowtie in the proper
  subinterval $[e,1]$, if $h \neq d$, then $(a,b:h,d)$ is a bowtie in
  the proper subinterval $[f,1]$, and one of these conditions holds
  because $c$ and $d$ are distinct.
\end{proof}

The following corollary restates Lemma~\ref{lem:atom-subint} as a
positive assertion.

\begin{cor}[Lattice induction]\label{cor:lattice-induct}
  If $P$ is a discretely graded bounded poset in which all atoms have
  joins and all proper subintervals are lattices, then $P$ itself is a
  lattice.
\end{cor}

In order to help investigate the lattice question in this context, the
first author wrote a program \texttt{euclid.sage} which is available
upon request.  Using this program we verified that these intervals are
lattices up through dimensions $9$ and we record this fact as a
proposition.

\begin{prop}[Low rank]\label{prop:low-rank}
  Let $w$ be a Coxeter element in an irreducible euclidean Coxeter
  group $\cox(\wt X_n)$.  If $n \leq 9$ then the interval $[1,w]^C$ is
  a lattice in the corresponding crystallographic group.
\end{prop}

Since all five sporadic examples of irreducible euclidean Coxeter
groups are covered by Proposition~\ref{prop:low-rank}, we may turn our
attention to the four infinite families.  Before considering joins of
atoms in the intervals for the infinite euclidean families, it might
be useful to consider the properties of atomic joins in the Coxeter
intervals of the most classical spherical family.

\begin{rem}[Atomic joins in the symmetric group]\label{rem:atom-sym}
  If $W$ is the symmetric group, i.e. the spherical Coxeter group of
  type $A$, then its Coxeter element is an $n$-cycle and the interval
  $[1,w]^W$ is the lattice of noncrossing partitions.  The atoms in
  this case are the transpositions and these are represented as
  boundary edges or diagonals in the corresponding convex $n$-gon.
  Notice that the join of two atoms always has very low rank: it is
  reflection length $2$ or $3$ regardless of $n$.  It has length $2$
  when the edges are noncrossing or share an endpoint and it has
  length $3$ when they cross.  In all three situations the join is
  below the element that corresponds to the triangle or square which
  is the convex hull of the union of their endpoints.
\end{rem}

The situation in the infinite euclidean families is very similar in
the sense that joins of atoms are of uniformly low rank and they live
in subposets defined by the endpoints, or equivalently the
coordinates, involved.  The first crucial fact is that there is a
well-defined projection from the middle row to the top and from the
middle row to the bottom row.

\begin{lem}[Projection]\label{lem:projection}
  Let $w$ be a Coxeter element in an irreducible euclidean Coxeter
  group $W = \cox(\wt X_n)$.  For each element $u$ in the middle row
  of the coarse structure of $[1,w]^W$, the set of elements in the top
  row that are above $u$ have a unique minimum element.  Similarly,
  the set of elements in the bottom row that are below $u$ have a
  unique maximum element.
\end{lem}

\begin{proof}
  For the five sporadic examples and the beginnings of the infinite
  families, we verified these assertions using the program
  \texttt{euclid.sage}.  Next we consider the elements in the first
  box of the middle row, the ones corresponding to vertical
  reflections.  Because of the explicit and regular nature of the
  infinite families (as illustrated by the computations given in
  \cite[Section~$11$]{Mc-lattice}), the list of top row elements above
  each vertical reflection can be explicitly written down and a unique
  minimal top element identified.  In type $A$, regardless of choice
  of Coxeter element, each vertical reflection is below a unique top
  row element in first box (i.e. a pure translation).  In type $C$,
  some vertical translations project upwards to elements in the first
  box of the row and other to the second.  In type $B$, each vertical
  translation projects upwards to a unique element in either the
  first, the second or the third box in the top row.  And in type $D$,
  each vertical translation projects upwards to a unique element in
  either the first or the fourth box in the top row.  

  Finally, let $u$ be an arbitrary element of the middle row and let
  $a$ be one of the vertical reflections below $u$.  Such a reflection
  must exists in any factorization of $u$ because, by definition of
  the middle row, some point experiences a vertical motion under $u$.
  We claim that the unique minimum top row element above $u$ is the
  join of $u$ and the projection of $a$ to the top row inside the
  interval $[a,w]^W$.  Because $a$ is a vertical reflection, its
  complement is also a vertical elliptic isometry and the interval
  $[a,w]^W$ is that of spherical type, thus a lattice, and so the join
  of these two elements is well-defined.  This element is clearly in
  the top row (because it is above the upward projection of $a$) and
  above $u$.  It is the minimum such element because any $v$ in the
  top row that is above $u$ is also above $a$, thus above the upward
  projection of $a$, and so above the join of $u$ and the upward
  projection of $a$.  The second assertion follows immediately from
  the first because these posets are self-dual
  (Proposition~\ref{prop:balanced}).
\end{proof}

Using Lemma~\ref{lem:projection} we define an \emph{upward projection
  map} from $[1,w]^\ccryst$ to $[1,w]^\cfac$ which is the identity on
$[1,w]^\cfac$ and sends elements in the middle row to the elements
described in the lemma.  It can be used to show that the meets and
joins that exist in the factor interval $[1,w]^\cfac$ remain meets and
joins inside the crystallographic interval $[1,w]^\ccryst$.

\begin{lem}[Factor meets and joins]\label{lem:factor-join}
  For each choice of Coxeter element $w$ in an irreducible euclidean
  Coxeter group $W = \cox(\wt X_n)$, the inclusion of the factor
  lattice $[1,w]^\cfac$ into the crystallographic interval
  $[1,w]^\ccryst$ preserves meets and joins.  In particular, any two
  elements in $[1,w]^\cfac$ have a well-defined meet in
  $[1,w]^\ccryst$ that agrees with their meet in $[1,w]^\cfac$ and a
  well-defined join in $[1,w]^\ccryst$ that agrees with their join in
  $[1,w]^\cfac$
\end{lem}

\begin{proof}
  Let $P = [1,w]^\ccryst$ be the crystallographic interval, let $Q =
  [1,w]^\cfac$ be the factor subposet and suppose that $u$ and $v$ are
  elements in $Q$ with a maximal lower bound $a$ in $P$ that is not
  their meet $b = u \meet_Q v$ in $Q$.  If $a$ is in $Q$ then $a=b$
  because $Q$ is a lattice, in particular a product of type $B$
  noncrossing partiition lattices.  Thus $a$ is not in $Q$ and must
  lie in the middle row of the coarse structure.  This means that $u$
  and $v$, being both above $a$ and in $Q$, must both lie in the top
  row.  By Lemma~\ref{lem:projection} there is a unique minimum top
  row element $c$ above $a$ which would, by definition, be below both
  $u$ and $v$, contradicting the maximality of $a$ as a lower bound
  for these elements.  Thus no such $u$ and $v$ exist. The assertion
  involving joins is true by duality.
\end{proof}

Lemma~\ref{lem:projection} can also be used to show that joins with
factored translations are well-defined.

\begin{lem}[Translation joins]\label{lem:trans-join}
  Let $w$ be a Coxeter element in an irreducible euclidean Coxeter
  group $W = \cox(\wt X_n)$.  If $a$ and $b$ are atoms in the
  crystallographic interval $[1,w]^\ccryst$ and one of them is a
  factored translation then their join is well-defined.
\end{lem}

\begin{proof}
  Let $b \in T_F$ be the factored translation.  If $a$ is in $\cfac$
  then by Lemma~\ref{lem:factor-join} the join of $a$ and $b$ is
  well-defined.  The only remaining case is where $a$ is in the middle
  row of the coarse structure and we claim that the join of $b$ with
  the upward projection of $a$ to the top row
  (Lemma~\ref{lem:projection}) is the join of $a$ and $b$.  In this
  case, the only upper bounds for $a$ and $b$ are to be found in the
  top row of the coarse structure and any such element is above the
  projection of $a$ by definition and thus above its join with $b$.
  This completes the proof.
\end{proof}

And finally, we consider the case where both atoms are reflections.

\begin{lem}[Reflection joins]\label{lem:refl-join}
  Let $w$ be a Coxeter element in an irreducible euclidean Coxeter
  group $W = \cox(\wt X_n)$.  If $a$ and $b$ are reflections in the
  interval $[1,w]^\ccryst$ and $a$ is a vertical reflection then their
  join is well-defined.
\end{lem}

\begin{proof}
  If $a$ and $b$ have no upper bounds in the middle row of the coarse
  structure then their join is the join of their images under the
  upward projection map by Lemma~\ref{lem:projection} and
  Lemma~\ref{lem:factor-join}.  If $W$ is of sporadic type then the
  join of $a$ and $b$ exists by Proposition~\ref{prop:low-rank}.  And
  finally, if $W$ belongs to one of the infinite euclidean families,
  one can use properties of the noncrossing partition lattices in the
  spherical infinite families, and properties of the upward projection
  map to show that every possible minimal upper bound for $a$ and $b$
  is below a low-rank top row element solely defined by the set of
  coordinates involved in the roots of $a$ and $b$ and the type of
  $W$.  This is the euclidean analogue of the situation described in
  Remark~\ref{rem:atom-sym}.  In other words, if there is a pair of
  reflection atoms in a crystallographic interval for one of the
  infinite families that has no well-defined join, then there is such
  a pair in such an interval where the rank is low and uniformly
  bounded.  And since no such pair exists in low rank
  (Proposition~\ref{prop:low-rank}), no such pair exists at all.
\end{proof}

Combining these lemmas establishes the following.

\begin{thm}[Lattice]\label{thm:lattice}
  For each choice of Coxeter element $w$ in an irreducible euclidean
  Coxeter group $W = \cox(\wt X_n)$, the crystallographic interval
  $[1,w]^\ccryst$, in the corresponding crystallographic group
  $C=\cryst(\wt X_n)$, is a lattice.
\end{thm}

\begin{proof}
  Proposition~\ref{prop:low-rank} covers the five sporadic examples.
  and for the four infinite families we proceed by induction.  The
  base cases are again covered by Proposition~\ref{prop:low-rank}, so
  suppose by induction that $X$ is $A$, $B$, $C$ or $D$ and that all
  crystallographic intervals are lattices for $k < n$.  Atoms in
  $[1,w]^\ccryst$ correspond to elements in $R_H \cup R_V \cup T_F$
  and all possible combinations of pairs of atoms are covered by
  Lemma~\ref{lem:factor-join}, Lemma~\ref{lem:trans-join}, or
  Lemma~\ref{lem:refl-join}.  Thus all pairs of atoms have
  well-defined joins and the interval is a lattice by
  Corollary~\ref{cor:lattice-induct}.
\end{proof}

Theorem~\ref{thm:lattice} and Proposition~\ref{prop:balanced} show
that Proposition~\ref{prop:lattice-garside} can be applied and this
immediately proves the following slightly more explicit version of
Theorem~\ref{main:garside}.

\begin{thm}[Crystallographic Garside groups]\label{thm:cryst-gar}
  Let $w$ be a Coxeter element in an irreducible euclidean Coxeter
  group $W = \cox(\wt X_n)$ and let $C = \cryst(\wt X_n,w)$ be the
  corresponding crystallographic group with its natural weighted
  generating set.  The interval $[1,w]^C$ is a balanced lattice and,
  as a consequence, it defines an interval group $\ggar = \gar(\wt
  X_n,w)$ with a Garside structure of infinite type.
\end{thm}

\section{Proof of Theorem B: Dual Artin subgroups}

In this section we prove Theorem~\ref{main:subgroup} by showing that
the Garside group $\gar(\wt X_n,w)$ is an amalgamated free product
with the dual Artin group $\dart(\wt X_n,w)$ as one of its factors.
The proof begins by noting the immediate consequences of
Lemma~\ref{lem:int-rel} on the level of presentations.

\begin{lem}[Presentation]\label{lem:presentation}
  For each choice of Coxeter element $w$ in an irreducible euclidean
  Coxeter group $W = \cox(\wt X_n)$, the Garside group $\ggar =
  \gar(\wt X_n,w)$ has a presentation whose generators and relations
  are obtained as a union of the generators and relations for
  presentations for $\gdiag$, $\gfac$ and $\gart$.
\end{lem}

\begin{prop}[Pushout]\label{prop:pushout}
  For each irreducible euclidean Coxeter group and for each choice of
  Coxeter element $w$, the Garside group $\ggar$ is the pushout of the
  diagram $\gfac \leftarrow \gdiag \rightarrow \gart$.  If the maps from
  $\gdiag$ to $\gfac$ and $\gart$ are both injective, then $\ggar$ is
  an amalgamated free product of $\gfac$ and $\gart$ over $\gdiag$
  and, in particular, $\gart$ injects into $\ggar$.
\end{prop}

We now show that these maps are injective.

\begin{lem}[$\ghor \into \gfac$]\label{lem:h-in-f}
  For each irreducible euclidean Coxeter group and for each choice of
  Coxeter element $w$, the horizontal group $\ghor$ injects into the
  factorable interval group $\gfac$.  As a consequence, the horizontal
  group $\ghor$ also injects into the diagonal interval group
  $\gdiag$.
\end{lem}

\begin{proof}
  The first assertion is a consequence of
  Proposition~\ref{prop:hor-maps} applied to each factor and the
  second assertion follows immediately since $\ghor \into \gfac$
  factors through $\gdiag$.
\end{proof}

\begin{lem}[$\gdiag \into \gfac$]\label{lem:d-in-f}
  For each irreducible euclidean Coxeter group and for each choice of
  Coxeter element $w$, the diagonal interval group $\gdiag$ injects
  into the factorable interval group $\gfac$.
\end{lem}

\begin{proof}  
  Recall that $R_H \cup \{w\}$ is one possible generating set for the
  diagonal interval group $\gdiag$ (Remark~\ref{rem:diag}) and let $U$
  be a word in these generators that represents an element $u \in
  \gdiag$.  If $u$ is in the kernel of the map $\gdiag \to \gfac$,
  then $u$ is also in the kernel of the composite map $\gdiag \to
  \gfac \to \cfac \to \Z$ where the middle map is the natural
  projection and the final map to $\Z$ is the vertical displacement
  map.  Since this composition sends each horizontal reflection to $0$
  and each $w$ to a nonzero integer, we conclude that the exponent sum
  of the $w$'s inside $U$ is zero.  Using the relations in $\gdiag$
  which describe how $w$ conjugates the elements of $R_H$, we can then
  find a word $U'$ with no $w$'s which still represents $u$ in
  $\gdiag$.  This means that $u$ is in the subgroup generated by
  elements of $R_H$ which by Lemma~\ref{lem:h-in-f} we can identify
  with $\ghor$.  In particular, $u$ is in the kernel of the map $\ghor
  \to \gfac$ which is trivial by Lemma~\ref{lem:h-in-f} proving
  $\gdiag \into \gfac$.
\end{proof}

\begin{lem}[$\gdiag \into \gart$]\label{lem:d-in-a}
  For each irreducible euclidean Coxeter group and for each choice of
  Coxeter element $w$, the factorable interval group $\gfac$ injects
  into the Garside group $\ggar$.  As a consequence, the diagonal
  interval group $\gdiag$ injects into the Artin group $\gart$.
\end{lem}

\begin{proof}
  The interval $[1,w]^\cfac$ is a lattice because it is a product of
  type $B$ partition lattices and $[1,w]^\ccryst$ is a lattice by
  Theorem~\ref{thm:cryst-gar}.  That they are balanced follows
  immediately from the fact that the corresponding generating sets in
  $\cfac$ and $\ccryst$ are closed under conjugation.  Finally, by
  Lemma~\ref{lem:factor-join} the inclusion of the former into the
  latter preserves meets and joins.  Thus by
  Proposition~\ref{prop:inj} the induced map from $\gfac$ to $\ggar$
  is injective.  Since $\gdiag \into \gfac$ by Lemma~\ref{lem:d-in-f},
  the composition injects $\gdiag$ into $\ggar$.  But $\gdiag \into
  \ggar$ factors through $\gart$, so the map $\gdiag \to \gart$ is
  also one-to-one.
\end{proof}

Proposition~\ref{prop:pushout} combined with Lemmas~\ref{lem:d-in-f}
and~\ref{lem:d-in-a} immediately prove the following slightly more
explicit version of Theorem~\ref{main:subgroup}.

\begin{thm}[Amalgamated free product]\label{thm:amalgamated}
  For each irreducible euclidean Coxeter group $\cox(\wt X_n)$ and for
  each choice of Coxeter element $w$, the Garside group $G = \gar(\wt
  X_n,w)$ can be written as an amalgamated free product of $W_w$ and
  $\gfac$ amalgamated over $\gdiag$ where $W_w$ is the dual Artin
  group $\dart(\wt X_n,w)$, $\gfac$ is the factorable inteval group,
  and $\gdiag$ is the diagonal interval group.  As a consequence, the
  dual Artin group $W_w$ injects into the Garside group $G$.
\end{thm}

Note that when the horizontal root system has only a single component,
$T \cong T_F$, $\gdiag \cong \gfac$ and $W_w \cong \ggar$.  This
occurs in types $C$ and $G$ and in type $A$ when $q=1$.

\section{Proof of Theorem C: Naturally isomorphic groups}
In this section we prove that the dual Artin group $\dart(\wt X_n,w)$
is isomorphic to the Artin group $\art(\wt X_n)$.  The first step is
to find homomorphisms between them.  In one direction this is easy to
do.

\begin{prop}[$A \onto \gart$]\label{prop:onto-dart}
  For every irreducible euclidean Coxeter group $W = \cox(\wt X_n)$
  and for each choice of Coxeter element $w$ as the product of the
  standard Coxeter generating set $S$, there is a natural map from the
  Artin group $A = \art(\wt X_n)$ onto the dual Artin group $\gart =
  \dart(\wt X_n,w)$ which extends the identification of the generators
  of $A$ with the subset of generators of $\gart$ indexed by $S$.
\end{prop}

\begin{proof}
  For every pair of elements in $S$, there is a rewritten
  factorization of $w$ where they occur successively and then the
  Hurwitz action on this pair produces the dual dihedral Artin
  relations corresponding to the angle between these two facets of
  $\sigma$ (Example~\ref{ex:dihedral}).  Systematically eliminating
  the other variables shows that these two elements in the dual Artin
  group satisfy the appropriate Artin relation.  This shows that the
  function injecting the generating set of the Artin group into the
  dual Artin group extends to a group homomoprhism.  The fact that it
  is onto is a consequence of the transitivity of the Hurwitz action.
\end{proof}

\begin{rem}[A $\wt G_2$ map]
  As an example of such a homomorphism, consider the simplex in the
  $\wt G_2$ tiling bounded by the lines $c_0$, $a_1$ and $d_1$ with
  bipartite Coxeter element $w = a_1 d_1 c_0$ in the notation of
  Definition~\ref{def:g2-gens}.  Proposition~\ref{prop:onto-dart}
  gives a homomorphism from the Artin group $\art(\wt G_2)$ with
  generators that we call $a$, $c$ and $d$ satisfying the relations
  $aca=cac$, $ad=da$ and $cdcdcd=dcdcdc$ to the dual Artin group
  $\dart(\wt G_2,w)$ extending the map sending $a$, $c$ and $d$ to
  $a_1$, $c_0$ and $d_1$, respectively.
\end{rem}

Defining a homomorphism in the other direction is more difficult
because we need to describe where the infinitely many generators are
to be sent and we need to check that infinitely many dual braid
relations are satisfied.  The first step is to describe certain
portions of the Cayley graph of an irreducible Artin group that are
already well understood.  These are portions of the Coxeter group
Cayley graph that lift to the Artin group.

\begin{defn}[Cayley graphs and Coxeter groups]
  The standard way to view the right Cayley graph of an irreducible
  euclidean Coxeter group with respect to a Coxeter generating set $S$
  is to consider the cell complex dual to the Coxeter complex.  The
  dual complex for the $\wt A_2$ Coxeter group, for example, is a
  hexagonal tiling of $\R^2$.  The dual complex has one vertex for
  each chamber of the Coxeter complex (and thus one vertex for each
  element of $W$) and it is convenient to place this vertex at the
  center of the insphere of this simplex so that it is equidistant
  from each facet.  Once labels are added to the edges of the
  $1$-skeleton of the dual cell complex, this becomes either the full
  right Cayley graph of $W$ with respect a simple system $S$, or it is
  a portion of the left Cayley graph with respect to the set of all
  reflections.  To get the full right Cayley graph we label the edges
  leaving a particular chamber $\sigma$ and then propigate the labels
  so that they are invariant under the group action.  To get a portion
  of the left Cayley graph we label the edges dual to the facets of
  the simplices by the unique hyperplane the facet determines.
\end{defn}

Converting between left Cayley graph labels and right Cayley graph
labels is a matter of conjugation.

\begin{rem}[Converting Labels]
  Suppose that we have picked a vertex corresponding to a chamber as
  our basepoint and indexed the vertices by the unique group element
  in $W$ which takes our base vertex to this vertex and suppose
  further that $a$, $c$ and $d$ are part of the standard generating
  set leaving our base vertex $v_1$.  In the right Cayley graph the
  edge connecting the adjacent vertices $v_{ac}$ and $v_{acd}$ is
  labeled by $d$ but in the corresponding portion of the left Cayley
  graph, its label is the reflection $(ac)d(ac)^{-1}$.  This is
  because this is the reflection we multiply by on the left to get
  from $ac$ to $acd$.  Geometrically we are conjugating the label in
  the right Cayley graph by the path in the right Cayley graph from
  $v_1$ to its starting vertex.
\end{rem}

There are a variety of ways that the unoriented right Cayley graph for
an irreducible euclidean Coxeter group can be converted into a portion
of the right Cayley graph for the corresponding Artin group.  We
describe two such procedures.

\begin{defn}[Standard flats]
  Let $W = \cox(\wt X_n)$ be an irreducible euclidean Coxeter group
  and let $A = \art(\wt X_n)$ be the corresponding Artin group.  If we
  pick a vector $\gamma$ that is generic in the sense that none of the
  roots of the hyperplanes of $W$ are orthogonal to $\gamma$, then we
  can orient the edges of the right Cayley graph of $W$ (which are
  transverse to the hyperplanes) according to the direction that forms
  an acute angle with $\gamma$.  Such a Morse function turns the
  boundary of every $2$-cell in the dual cell complex into an Artin
  relation.  In particular, the $2$-skeleton of the dual cell complex
  is simply connected and its labeled oriented $1$-skeleton is a
  portion of the right Cayley graph of $A$ that we call a
  \emph{standard flat}.  The terminology reflects the fact that the
  polytopes in the dual cell complex with labelled oriented edges can
  be added to the presentation complex for the Artin group $A$ without
  changing its fundamental group.  The universal cover of the result
  is known as the \emph{Salvetti complex} \cite{Sa87,Sa94}.  If each
  polytope is given the natural euclidean metric that it inherits,
  then a standard flat represents the $1$-skeleton of a metric copy of
  $\R^n$ inside the Salvetti complex.
\end{defn}

An easy way to create a standard flat is to let $\gamma$ be a generic
perturbation of the direction of the Coxeter axis.  What we really
need is a slight variation of this procedure.

\begin{defn}[Axial flats]
  Let $W = \cox(\wt X_n)$ be an irreducible euclidean Coxeter group
  with Coxeter element $w$ and let $A = \art(\wt X_n)$ be the
  corresponding Artin group.  Orient the edges of the dual cell
  complex as follows.  For hyperplanes that cross the Coxeter axis,
  orient the transverse edges so that their direction vector forms an
  acute angle with the direction of the Coxeter axis.  For the other
  hyperplanes with horizontal normal vectors, orient the transverse
  edges to point to the side that does not contain the Coxeter axis.
  We call such an oriented $1$-skeleton an \emph{axial flat}.  As
  before every $2$-cell in the dual cell complex has a boundary
  labelled by an Artin relation so this simply-connected $2$-complex
  lives in the Salvetti complex for $A$.  In fact, it is easy to see
  that it can be constructed by assembling sectors of standard flats
  around the column containing the Coxeter axis.
\end{defn}

Next we use axial flats to define reflections in euclidean Artin
groups.

\begin{defn}[Facets and reflections]\label{def:facet}
  Let $W = \cox(\wt X_n)$ be an irreducible euclidean Coxeter group
  with Coxeter element $w$ and fix a simplex $\sigma$ in the Coxeter
  complex or, equivalently, fix a vertex in the dual cell complex.
  For each facet of each simplex in the Coxeter complex we define a
  reflection in the corresponding Artin group $A = \art(\wt X_n)$ as
  follows.  Orient the edges of the dual cell complex so that it is
  the axial flat for $w$ and then conjugate the labelled oriented edge
  transverse to the specified facet by a path in the axial flat from
  the fixed basepoint to the start of the transverse edge.  
\end{defn}

Many of the facets belonging to a common hyperplane determine the same
reflection in the Artin group but describing which ones are equal is
slightly subtle.  

\begin{lem}[Consistency]\label{lem:consistent}
  Let $W = \cox(\wt X_n)$ be an irreducible euclidean Coxeter group
  with Coxeter element $w$ and a fixed base simplex.  Let $H$ be a
  hyperplane in the Coxeter complex, let $P$ be convex hull of the
  axial vertices in $H$ and suppose that $P$ contains at least one
  facet of a chamber.  If $\sigma_1$ and $\sigma_2$ are simplices on
  the same side of $H$ and $P \cap \sigma_i$ is a facet of $\sigma_i$
  for $i= 1,2$, then the reflections $r_1$ and $r_2$ that they define
  in the axial flat are equal in the Artin group $A=\art(\wt X_n)$.
\end{lem}

\begin{proof}
  The idea of the proof is straightforward.  Let $p$ be a path in the
  axial flat from the fixed base simplex to $\sigma_1$ and let $q$ be
  a path from $\sigma_1$ to $\sigma_2$ (also in the axial flat) that
  is as short as possible.  By construction $r_1 = (p) s_1 (p)^{-1}$
  and $r_2 = (p q) s_2 (p q)^{-1}$ for appropriate standard generators
  $s_1$ and $s_2$.  Because $q$ is as short as possible, it only
  crosses the hyperplanes that separate $\sigma_1$ from $\sigma_2$ and
  by Lemma~\ref{lem:convexity} this only includes hyperplanes whose
  normal vectors do not change sign in the axial flat when reflected
  across the hyperplane $H$.  In particular, the path $q s_2 q^{-1}
  s_1^{-1}$ is visible as a closed loop in the axial flat.  As a
  consequence it is trivial in $A$ and this relation shows that the
  elements $r_1$ and $r_2$ are equal.
\end{proof}

The necessity of the specificity given in Lemma~\ref{lem:consistent}
can be seen even in the $\wt G_2$ case.  We continue to use the
notation of Definition~\ref{def:g2-gens}.

\begin{rem}[Consistency]
  Consider the four line segments of the hyperplane $e_3$ inside the
  lightly shaded strip of Figure~\ref{fig:g2-axis}.  The reflections
  in $\art(\wt G_2)$ that they determine are $(d)c(d)^{-1}$,
  $(dac)a(dac)^{-1}$, $(dacd)a(dacd)^{-1}$ and
  $(dacdca)c(dacdca)^{-1}$.  All four belong to the convex hull of the
  axial vertices in the $e_3$ hyperplane and it is straightforward to
  show that all four expressions represent the same group element in
  $\art(\wt G_2)$.  On the other hand, consider the two line segments
  of the $c_2$ hyperplane inside the lightly shaded strip.  The
  reflections in $\art(\wt G_2)$ that they determine are
  $(ad)c(ad)^{-1}$ and $(dc)a(dc)^{-1}$.  The first is bounded by two
  axial vertices, the second is not and these two reflections are not
  equal in the Artin group.
\end{rem}

Fortunately the level of consistency available is sufficient to
establish the homomorphism we require.

\begin{defn}[Dual reflections in the Artin group]\label{def:dual-reflections}
  Let $W = \cox(\wt X_n)$ be an irreducible euclidean Coxeter group
  with Coxeter element $w$ and a fixed base simplex.  For each
  reflection labeling an edge in the interval $[1,w]^W$ we define an
  element of the Artin group $A = \art(\wt X_n)$ as follows.  When the
  axial vertices in the fixed hyperplane of a reflection $r$ have a
  convex hull which contains a facet of a simplex, we define the
  corresponding reflection in $A$ as described in
  Definition~\ref{def:facet}.  By Lemma~\ref{lem:consistent} the
  element defined is independent of the facet in the convex hull that
  we use.  This applies to all vertical reflections and to those
  horizontal reflections which contain a facet of the boundary of the
  convex hull of all axial vertices.  We call these the \emph{standard
    horizontal reflections}.  For the nonstandard horizontal
  reflections we proceed as follows.  By
  Proposition~\ref{prop:products} the subgroup $H_w$ generated by the
  horizontal reflections can be identified with a product of $k$
  euclidean braid groups.  From this identification it is clear that
  the standard horizontal reflections generate.  Next, in the axial
  flat we can see that the reflections in $A$ corresponding to the
  standard horizontal reflections satisfy the Artin relations
  associated with the dihedral angles between their hyperplanes.  This
  means that there is a natural homomorphism from the subgroup of
  $H_w$ to the subgroup generated by the images of the standard
  horizontal reflections in $A$.  We use this map to define the images
  of the nonstandard horizontal reflections in~$A$.
\end{defn}

\begin{prop}[Pure Coxeter element]\label{prop:pure-cox-elt}
  Let $W = \cox(\wt X_n)$ be an irreducible euclidean Coxeter group
  with Coxeter element $w$ and a fixed base simplex.  If $w^p$ is the
  smallest power of $w$ which acts on the Coxeter complex as a pure
  translation and $r$ is a standard horizontal reflection in the Artin
  group $A = \art(\wt X_n)$, then $w^p$ and $r$ (viewed as elements in
  $A$) commute.  As a consequence, $w^p$ centralizes the full subgroup
  of $A$ generated by these standard horizontal reflections.
\end{prop}

\begin{proof}
  This follows immediately from Lemma~\ref{lem:consistent}.  The
  convex hull of all axial vertices is, metrically speaking, a product
  of simplices cross the reals and the convex hull $P$ of the axial
  vertices contained in the fixed hyperplane of $r$ is one facet of
  this product of simplices cross the reals.  The entire configuration
  in the axial flat is invariant under the vertical translation
  induced by $w^p$ and thus $r$ and $(w^p) r (w^p)^{-1}$ define the
  same reflection in~$A$.
\end{proof}

We are now ready to define a homomorphism from the dual Artin group to
the Artin group.

\begin{prop}[$\gart \onto A$]\label{prop:onto-art}
  For every irreducible euclidean Coxeter group $W = \cox(\wt X_n)$
  for every choice of Coxeter element $w$ as the product of the
  standard Coxeter generating set $S$, the map on generators described
  above extends to a group homomorphism from the dual Artin group
  $\gart = \dart(\wt X_n,w)$ onto the Artin group $A = \art(\wt X_n)$.
\end{prop}

\begin{proof}
  Let $\sigma$ be the chamber in the Coxeter complex of $W$ bounded by
  the fixed hyperplanes of the reflections indexed by $S$ and consider
  the function from the reflections in $[1,w]^W$ to $A$ which sends
  each reflection to the reflection in $A$ as defined in
  Definition~\ref{def:dual-reflections}.  We only need to show that
  this function extends to a homomorphism.  As mentioned in
  Definition~\ref{def:dual-relations} there are three types of dual
  braid relations in the interval $[1,w]^W$.  The ones indexed by the
  third box in the bottom row are relations among horizontal
  reflections and their satisfaction was described in
  Definition~\ref{def:dual-reflections}.  
  
  The ones indexed by the second box in the middle row are vertical
  elliptics which rotate around a codimension~$2$ subspace.  Since its
  right complement is also vertical elliptic, all the reflections in
  the factorization fix an axial vertex $v$ which belongs to some
  axial simplex $\sigma'$.  The reflections in the Artin group $A$
  that fix the facets of $\sigma'$ form an alternative simple system
  $S'$ for $A$.  Using an old result from van der Lek's thesis, the
  subset of elements of $S'$ that fix $v$ generate an Artin group
  which injects into $A$, in this case an Artin group of spherical
  type \cite{vdLek83}.  Using the known equivalence of dual and
  standard presentations for spherical Artin groups we see that these
  dual braid relations are satisfied by their images in~$A$.

  Finally, the ones indexed by the first box in the top row are the
  various ways to factor a pure translation $t$ in $W$ and these are
  described in Proposition~\ref{prop:trans}.  Using the Hurwitz action
  there is a factorization of $w$ in $A$ that maps to a horizontal
  factorization of $w$ in $W$.  In particular, there is an element $t$
  in $A$ that differs from $w$ by a product of (the images of)
  horizontal reflections and which has a factorization $t = r' r$ in
  $A$ into reflections where $r$ and $r' = (w^p) r (w^{-p})$ are
  defined by vertically shifted facets of simplices.  The first
  observation combined with Proposition~\ref{prop:pure-cox-elt} shows
  that this $t$ commutes with $w^p$ inside $A$.  If we define
  reflections $r_i = (w^{ip}) r (w^{-ip})$ as the reflections in $A$
  defined by the various vertical shifts of the facet that defines
  $r$, then $t = (w^{ip}) t (w^{-ip}) = (w^{ip}) r_1 r_0 (w^{-ip}) =
  r_{i+1} r_i$ shows that all of factorizations of $t$ in the interval
  $[1,w]^W$ are also satisfied in $A$.  Since all three types of dual
  braid relations are satisfied, the function on reflections extends
  to a homomorphism, and this homomorphism is onto because its image
  includes a generating set for the Artin groups $A$.
\end{proof}

Our third main result now follows as a easy corollary.

\setcounter{main}{\mainisomorphic}
\begin{main}[Naturally isomorphic groups]
  For each irreducible euclidean Coxeter group $W = \cox(\wt X_n)$ and
  for each choice of Coxeter element $w$ as the product of the
  standard Coxeter generating set $S$, the Artin group $A=\art(\wt
  X_n)$ and the dual Artin group $W_w = \dart(\wt X_n,w)$ are
  naturally isomorphic.
\end{main}

\begin{proof}
  Let $\sigma$ be the chamber in the Coxeter complex for $W$ whose
  facets index the reflections in $S$.  Because $w$ is obtained as a
  product of the elements in $S$, every vertex of $\sigma$ is an axial
  vertex and all of $\sigma$ is contained in the convex hull of the
  axial vertices.  By composing the surjective homomorphisms described
  in Propositions~\ref{prop:onto-art} and~\ref{prop:onto-dart} we find
  a map from $A$ to itself which must be the identity homomorphism
  since it fixes each element of the generating set $S$.  This means
  the first map in the composition from $A$ to $W_w$ is injective as
  well as surjective and thus an isomorphism.
\end{proof}

\section{Proof of Theorem D: Euclidean Artin groups}

In a recent survey article Eddy Godelle and Luis Paris highlighted how
little we know about general Artin groups by stating four basic
conjectures that remain open \cite{GoPa-basic}.  Their four
conjectures are:
\begin{itemize}
  \item[(A)] All Artin groups are torsion-free.
  \item[(B)] Every non-spherical irreducible Artin group has a trivial
    center.
  \item[(C)] Every Artin group has a solvable word problem.
  \item[(D)] All Artin groups satisfy the $K(\pi,1)$ conjecture.
\end{itemize}
Godelle and Paris also remark that these conjectures remain open and
are a ``challenging question'' even in the case of the euclidean Artin
groups.  These are precisely the conjectures that we set out to
resolve.  In this section we prove our final main result,
Theorem~\ref{main:artin}, which resolves the first three of these
questions for euclidean Artin groups.  Most of the structural
properties follows from the existence of a classifying space which is
itself an easy corollary of Theorems~\ref{main:subgroup}
and~\ref{main:isomorphic}.

\begin{prop}[Classifying space]\label{prop:classifying-space}
  Every irreducible Artin group of euclidean type is the fundamental
  group of a finite dimensional classifying space.
\end{prop}

\begin{proof}
  By Theorem~\ref{thm:consequences}, the Garside group $\gar(\wt
  X_n,w)$ has a finite-dimensional classifying space and the cover of
  this space corresponding to the subgroup $\art(\wt X_n)$ is a
  classifying space for the Artin group.
\end{proof}

\begin{rem}[Finite-dimensional]
  The reader should note that the spaces involved are
  finite-dimensional but not finite.  More precisely, because the
  interval $[1,w]^C$ has infinitely many elements, the natural
  classifying space constructed for $\gar(\wt X_n,w)$ has infinitely
  many simplices, but their dimension is nevertheless bounded above by
  the combinatorial length of the longest chain.
\end{rem}

To compute the center of $\art(\wt X_n)$ we recall an elementary
observation about euclidean isometries which quickly leads to the
well-known fact that irreducible euclidean Coxeter groups are
centerless.

\begin{lem}[Coxeter groups]\label{lem:cox-center}
  Let $W = \cox(\wt X_n)$ be an irreducible euclidean Coxeter group,
  let $u\in W$ be an elliptic isometry and let $v\in W$ be a
  hyperbolic isometry.  If $\lambda$ is the translation vector of $v$
  on $\ms(v)$ and $\fix(u)$ is not invariant under $\lambda$, then $u$
  and $v$ do not commute.
\end{lem}

\begin{proof}
  Because $v$ lives in $W$, there is a power $v^m$ that is a pure
  translation with translation vector $m\lambda$.  If $u$ commutes
  with $v$ then $u$ commutes with $v^m$ but the fixed set of the
  conjugation of $u$ by $v^m$ is the translation of the fixed set of
  $u$ by $m \lambda$, contradiction.
\end{proof}

\begin{cor}[Trivial center]\label{cor:cox-center}
  Every irreducible euclidean Coxeter group has a trivial center.
\end{cor}

\begin{proof}
  Using the criterion of Lemma~\ref{lem:cox-center}, it is easy to
  find a noncommuting hyperbolic for each elliptic in $W$ and a
  noncommuting elliptic for each hyperbolic in $W$.
\end{proof}

We note one quick consequence for Artin groups.

\begin{lem}[Powers of $w$]\label{lem:w-powers}
  For each irreducible euclidean Coxeter group and for each choice of
  Coxeter element $w$, the nontrivial powers of $w$ are not central in
  the Artin group $\gart$.
\end{lem}

\begin{proof}
  For each nonzero integer $m$, the element $w^m$ projects to a
  nontrivial hyperbolic element in $\ccox$. By
  Corollary~\ref{cor:cox-center} there is an element $u$ in $\ccox$
  that does not commute with $w^n$ and because the projection map
  $\gart \onto \ccox$ is onto, it has preimages that do not commute
  with $w^n$ in $\gart$.
\end{proof}

We also derive a second more substantial consequence.

\begin{lem}[$\ggar \leadsto \gfac$]\label{lem:g-to-f}
  For each irreducible euclidean Coxeter group and for each choice of
  Coxeter element $w$, the simples in $\ggar$ which commute with $w$
  are simples in $\gfac$.  As a consequence, the elements of $\ggar$
  which commute with $w$ are contained in the subgroup $\gfac$.
\end{lem}

\begin{proof}
  If $u$ is a simple in $\ggar$ which commutes with $w$ then the image
  of $u$ as a euclidean isometry commutes with the power $w^m$ whose
  image as a euclidean isometry in $\ccryst$ is a pure translation in
  the direction of the Coxeter axis.  When $u$ is elliptic, by
  Lemma~\ref{lem:cox-center} it has a vertically invariant fixed set,
  it does not belong to the middle row of the coarse structure, and
  thus $u \in [1,w]^\cfac$.  The extension from simples to elements
  follows from Proposition~\ref{prop:nf}.
\end{proof}

\begin{lem}[$\gfac \leadsto \Z^k$]\label{lem:f-to-z^k}
  For each irreducible euclidean Coxeter group and for each choice of
  Coxeter element $w$, the centralizer of $w$ in $\gfac$ is the group
  $\Z^k \cong \langle w_i \rangle$ where the $w_i$ are the special
  elements in the factors whose product is $w$.
\end{lem}

\begin{proof}
  By Proposition~\ref{prop:products}, the group $\gfac$ and the
  interval $[1,w]^{\gfac}$ split as direct products.  Thus the simples
  that commute with $w$ are products of the simples in each factor
  that commute with the factor $w_i$.  Since by
  Proposition~\ref{prop:commuting-with-w} such a simple in each factor
  must be $1$ or $w_i$, there are exactly $2^k$ simples that commute
  with $w$.  And since the $w_i$ commute with each other in $\gfac$
  they generate a subgroup isomorphic to $\Z^k \cong \langle w_i
  \rangle$, with one $\Z$ from each factor.  All of these commute with
  $w$ and by Proposition~\ref{prop:nf} these are the only elements
  that commute with $w$.
\end{proof}

\begin{lem}[$\Z^k \leadsto \Z$]\label{lem:z^k-to-z}
  For each irreducible euclidean Coxeter group and for each choice of
  Coxeter element $w$, the intersection of $\gdiag$ and the group
  $\Z^k$ (generated by the $w_i$ factors of $w$) is an infinite cyclic
  subgroup generated by $w$.  In symbols $\gdiag \cap \langle w_i
  \rangle \cong \langle w \rangle$.
\end{lem}

\begin{proof}
  Combining the global winding number maps for each factor
  (Definition~\ref{def:middle-maps}) produces a map $\gfac \to \Z^k$
  which sends $w_i$ to $e_i$, the $i$-th unit vector in $\Z^k$ which
  restricts to an isomorphism on the subgroup $\langle w_i \rangle$ in
  $\gfac$.  The image of $\gdiag$ under the composition $\gdiag \into
  \gfac \to \Z^k$ is the set of $k$-tuples with all coordinates equal.
  To see this we view an element of $\gdiag$ as a product of simples
  thought of as elements of $\cdiag$ rather than $\gdiag$.  The
  relevant maps are now vertical displacement maps rather than global
  winding number maps.  From this perspective it is clear that the
  horizontal reflections are sent to the zero vector under this
  composition and that every diagonal translation $t$ is sent to the
  vector with all coordinates equal to $1$.  Thus the only elements in
  the intersection are those with the same number of $w_i$'s for each
  $i$.  Using the fact that they commute with each other, we can thus
  rewrite this expression as a power of $w = \prod_i^k w_i$.
\end{proof}

And finally we put all the pieces together.

\begin{prop}[Center]\label{prop:center}
  An irreducible euclidean Artin group has a trivial center.
\end{prop}

\begin{proof}
  Let $\gart$ be a dual euclidean Artin group with special element
  $w$.  If $u$ is central in $\gart$ then $u$ commutes with $w$ and by
  Proposition~\ref{prop:nf}, $u$, viewed as an element of $\ggar$, has
  a Garside normal form built out of simples that commute with $w$.
  By Lemma~\ref{lem:g-to-f} the only such simples are simples in
  $\gfac$, so $u \in \gfac$ and by Lemma~\ref{lem:f-to-z^k} the
  element $u$ in fact belongs to the subgroup $\Z^k \cong \langle w_i
  \rangle$ generated by the special factors $w_i$ of $w$.  This means
  that $u$ is in $\gfac \cap \gart$ and thus in $\gdiag$ by the
  amalgamated free product structure of $\ggar$
  (Theorem~\ref{thm:amalgamated}).  But by Lemma~\ref{lem:z^k-to-z}
  the only portion of the $\Z^k \cong \langle w_i \rangle$ contained
  in $\gdiag$ is the subgroup $\Z \cong \langle w \rangle$.  In
  particular, $u = w^n$ for some $n$.  And finally, by
  Lemma~\ref{lem:w-powers} the nontrivial powers of $w$ are not
  central in $\gart$, so the center of $\gart$ is trivial.
\end{proof}

These combine to give our main result.

\setcounter{main}{\mainartin}
\begin{main}[Euclidean Artin groups]
  Every irreducible euclidean Artin group $\art(\wt X_n)$ is a
  torsion-free centerless group with a solvable word problem and a
  finite-dimensional classifying space.
\end{main}

\begin{proof}
  Because $\art(\wt X_n)$ is isomorphic to $\gart$ which is a subgroup
  of a Garside group $\ggar = \gar(\wt X_n,w)$, the standard solution
  to the word-problem in $\ggar$ gives a solution to the word problem
  in $\gart$ and by Proposition~\ref{prop:classifying-space} it has a
  finite dimensional classifying space.  Groups with
  finite-dimensional classifying spaces are torsion-free and by
  Proposition~\ref{prop:center} its center is trivial.
\end{proof}

The fourth question of Godelle and Paris, the $K(\pi,1)$ conjecture,
would have a positive resolution if one could establish the following.

\begin{conj}[Homotopy equivalence]
  The classifying space for each irreducible Artin group of euclidean
  type constructed here, should be homotopy equivalent to the standard
  topological space with this fundamental group constructed from the
  action of the corresponding Coxeter group on its complexified
  hyperplane complement.
\end{conj}

And finally, there is another obvious question to ask at this point,
although we suspect that it may have a negative answer.

\begin{quest}
  Is there a natural way to extend the definitions of $\cryst(\wt
  X_n,w)$ and $\gar(\wt X_n,w)$ to other infinite Coxeter groups so
  that they retain their key properties?  In particular, is every
  Artin group isomorphic to a subgroup of a suitably-defined Garside
  group?
\end{quest}

\def\cprime{$'$}
\providecommand{\bysame}{\leavevmode\hbox to3em{\hrulefill}\thinspace}
\providecommand{\MR}{\relax\ifhmode\unskip\space\fi MR }
\providecommand{\MRhref}[2]{%
  \href{http://www.ams.org/mathscinet-getitem?mr=#1}{#2}
}
\providecommand{\href}[2]{#2}

\end{document}